\documentclass[11pt]{article}
\usepackage{verbatim}
\usepackage{amssymb,amsmath}
\usepackage{amsthm}
\usepackage{graphicx}
\usepackage{epsfig}
\newtheorem{corollary}{Corollary}[section]
\newtheorem{lemma}[corollary]{Lemma}
\newtheorem{proposition}[corollary]{Proposition}
\newtheorem{theorem}[corollary]{Theorem}
\newcommand{\Prob} {{\bf P}}
\newcommand{\Z}{{\mathbb Z}}
\newcommand{\E}{{\bf E}}

\newcommand{\R}{{\mathbb{R}}}
\newcommand{\C}{{\mathbb C}}

\newcommand{\dist}{{\rm dist}}

\def \Im {{\rm Im}}
\def \Re {{\rm Re}}
\def \p {\partial}

\def \Half {{\mathbb H}}

\def \Disk {{\mathbb D}}

\def \hcap {{\rm hcap}}

\def \distsub {{\Upsilon}}

\def \F {{\cal F}}
\def  \G {{\cal G}}
\def \domain {{\cal D}}

\def \newtheta{\tilde \Theta}
\def \I {{\mathcal I}}

\def \dyadic {{\cal Q}}

\newenvironment{remark}[1][Remark]{\begin{trivlist}
\item[\hskip \labelsep {\bfseries #1}]}{\end{trivlist}}
\newenvironment{definition}[1][Definition]{\begin{trivlist}
\item[\hskip \labelsep {\bfseries #1}]}{\end{trivlist}}

\begin{document}

\title{The natural parametrization for the
Schramm-Loewner evolution}

\author{Gregory F. Lawler \thanks{Research supported by National
Science Foundation grant DMS-0734151.}\\
University of Chicago  \\ \\ \\
Scott Sheffield\thanks{Research supported by National Science
Foundation grants DMS-0403182, DMS-0645585 and
OISE-0730136.}\\Massachusetts Institute of Technology}

\maketitle

\begin{abstract}  The Schramm-Loewner evolution ($SLE_\kappa$)
is a candidate for the scaling limit of random curves arising in
two-dimensional critical phenomena.  When $\kappa < 8$, an instance
of $SLE_\kappa$ is a random planar curve with almost sure Hausdorff dimension $d = 1 + \kappa/8 < 2$.
This curve is conventionally parametrized by its half plane capacity, rather than by any measure
of its $d$-dimensional volume.

For $\kappa < 8$, we use a Doob-Meyer decomposition to construct the unique (under mild assumptions)
Markovian parametrization of $SLE_\kappa$ that transforms like a $d$-dimensional volume measure under conformal maps.
We prove that this parametrization is non-trivial (i.e., the curve is not entirely traversed
in zero time) for $\kappa < 4(7 - \sqrt{33}) = 5.021 \cdots$.
\end{abstract}

\section{Introduction}
\subsection{Overview}  \label{overviewsubsection} A number of measures on paths or clusters on two-dimensional
lattices arising from critical statistical mechanical
models are believed to exhibit some kind of conformal
invariance in the
scaling limit.
The Schramm-Loewner evolution ($SLE$ --- see Section \ref{SLEsubsec} for a definition) was created
by Schramm \cite{Schramm} as a candidate for the scaling limit
of these measures.

For each fixed $\kappa \in (0,8)$, an instance $\gamma$ of $SLE_\kappa$ is a random planar curve
with almost sure Hausdorff dimension $d = 1 + \kappa/8 \in (1,2)$ \cite{Beffara}.
This curve is conventionally parametrized by its half plane capacity (see Section \ref{SLEsubsec}), rather than by any measure
of its $d$-dimensional volume.  Modulo time parametrization, it has been shown that several
discrete random paths on grids (e.g., loop-erased random walk
\cite{LSWlerw}, harmonic explorer \cite{SSharm}) have $SLE$ as a scaling limit.  In these cases, one would
expect the natural discrete parametrization (in which each edge is traversed in the same amount of time)
of the lattice paths to scale to a continuum parametrization of $SLE$.  The goal of this paper is to construct
a candidate for this parametrization, a candidate which is (like $SLE$ itself) completely characterized by
its conformal invariance symmetries, continuity, and Markov properties.

When $\kappa \geq 8$ (and $SLE_\kappa$ is almost surely space-filling) the natural candidate is
the area parameter $\Theta_t :=\text{Area} \,\, \gamma([0,t])$.  One could
use something similar for $\kappa < 8$ if one could show that some conventional measure of the $d$-dimensional volume
of $\gamma([0,t])$ (e.g., the $d$-dimensional Minkowski content or some sort of Hausdorff content; see Section \ref{candidatesubsec}) was well-defined and non-trivial.  In that case, one could replace $\text{Area} \,\, \gamma([0,t])$ with the $d$-dimensional volume of
$\gamma([0,t])$.
We will take a slightly different route.  Instead of directly constructing a $d$-dimensional volume measure (using one
of the classical definitions), we will simply assume that there exists a locally finite measure on $\gamma$ that {\em transforms} like
a $d$-dimensional volume measure under conformal maps and then use this assumption (together with an argument based on the Doob-Meyer
decomposition) to deduce what the measure must be.
We conjecture that the measure we construct is equivalent to the $d$-dimensional Minkowski content, but we will not
prove this.  Most of the really hard work in this paper takes place in Section \ref{momentsec}, where certain second moment bounds
are used to prove that the measure one obtains from the Doob-Meyer decomposition is non-trivial (in particular, that it is
almost surely not identically zero). At present, we are only able to prove this for $\kappa < 4(7 - \sqrt{33}) = 5.021 \cdots$.

We mention that a variant of our approach, due to Alberts and the second co-author of this work,
appears in \cite{AlbertsSheffield}, which gives, for $\kappa \in (4,8)$, a natural {\em local} time parameter for the
intersection of an $SLE_\kappa$ curve with the boundary of the domain it is defined on.   The proofs in \cite{AlbertsSheffield}
cite and utilize the Doob-Meyer-based techniques first developed for this paper; however, the
second moment arguments in \cite{AlbertsSheffield} are very different
from the ones appearing in Section \ref{momentsec} of this work.  It is possible that our techniques will have other
applications.  In particular, it would be interesting to see whether
natural $d$-dimensional volume measures for other random $d$-dimensional sets with conformal invariance properties (such
as conformal gaskets \cite{SSW} or the intersection of an $SLE_{\kappa, \rho}$ with its boundary) can be constructed
using similar tools.  In each of these cases, we expect that obtaining precise second moment bounds will be the most difficult step.

A precise statement of our main results will appear in Section \ref{natparamsec}.  In the meantime, we present
some additional motivation and definitions.

We thank Brent Werness for useful comments on an earlier draft of this paper.

\subsection{Self avoiding walks: heuristics and motivation}
In order to further explain and motivate our main results, we include a
heuristic discussion of a single concrete example: the
self-avoiding walk (SAW).  We will not be very precise here; in fact, what we say
here about SAWs is still only conjectural.  All of the conjectural statements in this section
can be viewed as consequences of the ``conformal invariance Ansatz'' that is
generally accepted (often without a precise formulation) in the physics
literature on conformal field theory.
Let $D \subset \C$ be a simply connected bounded domain,
and let $z,w$ be distinct points on $\p D$.
Suppose that a lattice $\epsilon \Z^2$ is placed on $D$ and let
$\tilde z, \tilde w \in D$ be lattice points in $\epsilon
\Z^2$ ``closest'' to $z,w$.  A SAW  $\omega$
from
$\tilde z$ to $\tilde w$ is a sequence of distinct points
\[          \tilde z = \omega_0,\omega_1,\ldots,\omega_k=
    \tilde w , \]
with $\omega_j \in \epsilon \Z^2 \cap D$ and $|\omega_j
- \omega_{j-1}| = \epsilon$ for $1 \leq j \leq k$.  We write
$|\omega| = k$.
For each $\beta > 0$, we can consider the measure on SAWs from
$\tilde z$ to $\tilde w$ in $D$ that gives measure $e^{-\beta |\omega|},$
to each such
SAW.  There is a critical $\beta_0$, such that the
partition function
\[ \sum_{\omega: \tilde z \rightarrow
 \tilde w, \omega \subset D}  e^{-\beta_0 |\omega|} \]
neither grows nor decays exponentially as a function of
$\epsilon$ as $\epsilon \rightarrow 0.$   It is believed
that if we choose this $\beta_0$, and normalize so that this is
a probability measure, then there is a limiting measure on
paths that is the scaling limit.

It is further believed that the typical number of steps of a SAW in the measure
above is of order $\epsilon^{-d}$ where the exponent
$d = 4/3$ can be
considered the fractal dimension of the paths.  For fixed $\epsilon$,
let us define the scaled function
\[       \hat \omega ( {j} \epsilon^d )
                      =  \omega_{j} , \;\;\;\;
                           j=0,1,\ldots,|\omega|. \]
We use linear interpolation to make this a continuous
path $\hat \omega:[0,\epsilon^d|\omega|] \rightarrow \C$. Then
one expects that the following is true:
\begin{itemize}
\item  As $\epsilon \rightarrow 0$, the above probability measure
on paths converges to a probability measure $\mu_D^\#(z,w)$ supported
on continuous curves $\gamma:[0,t_\gamma] \rightarrow \C$ with
$\gamma(0) = z, \gamma(t_\gamma) = w, \gamma(0,t_\gamma) \subset
D$.
\item  The probability measures $\mu_D^\#(z,w)$ are conformally
invariant.  To be more precise, suppose $F: D \rightarrow D'$
is a conformal transformation that extends to $\p D$ at least
in neighborhoods of $z$ and $w$.  For each $\gamma$ in $D$
connecting $z$ and $w$, we will define a conformally transformed
path $F\circ \gamma$ (with a parametrization described below) on $D'$.
We then denote by $F \circ \mu_D^\#(z,w)$ the push-forward of the measure
$\mu_D^\#(z,w)$ via the map $\gamma \to F\circ \gamma$.  The conformal invariance assumption
is
\begin{equation} \label{confinvariance}
 F \circ \mu_D^\#(z,w) = \mu_{D'}^\#(F(z),F(w)). \end{equation}
\end{itemize}

Let us now define $F \circ \gamma$.  The path $F\circ \gamma$
will traverse the points $F(\gamma(t))$ in order; the only question
is how ``quickly'' does the curve traverse these points.  If we look
at how the scaling limit is defined, we can see that if $F(z)
= rz$ for some $r> 0$, then the lattice spacing $\epsilon$ on
$D$ corresponds to lattice space $r \epsilon$ on $F(D)$ and hence
we  would expect the time to traverse $r \gamma$
should be $r^d$ times the time to traverse $\gamma$.  Using this
as a guide locally, we say that
the amount of time needed to traverse $F(\gamma[t_1,t_2])$ is
\begin{equation}  \label{dec5.1}
        \int_{t_1}^{t_2 }
    |F'(\gamma(s))|^d \, ds .
\end{equation}
This tells us how to parametrize $F \circ \gamma$ and we include
this as part of the definition of $F \circ \gamma$.   This is analogous to the
known conformal invariance of Brownian motion in $\C$ where the
time parametrization must be defined as in \eqref{dec5.1} with
$d=2$.

If there is to be a family of probability measures $\mu_{D}^\#(z,w)$ satisfying
\eqref{confinvariance} for simply connected $D$, then we only need to define
$\mu_\Half^\#(0,\infty)$, where $\Half$ is the upper half plane.
To restrict the set of possible definitions, we introduce another property that one would
expect the scaling limit of SAW to satisfy.  The {\bf domain Markov property} states that if $t$ is a stopping time
for the random path $\gamma$ then given $\gamma([0,t])$,
the conditional law of the remaining path $\gamma'(s) := \gamma(t+s)$ (defined for $s \in [0, \infty)$) is
$$\mu_{\Half\setminus \gamma([0,t])}^\#(\gamma(t),\infty),$$ independent of the
parametrization of $\gamma([0,t])$.

If we consider $\gamma$ and $F \circ \gamma$ as being defined only up to reparametrization,
then Schramm's theorem states that \eqref{confinvariance} (here being considered as a statement about measures
on paths defined up to reparametrization) and the domain Markov property (again interpreted up to reparametrization) characterize the path
as being a chordal $SLE_\kappa$ for some $\kappa > 0$.  (In the case of the
self-avoiding walk, another property called the ``restriction
property'' tells us that we must have $\kappa = 8/3$ \cite{LSWrest,LSWsaw}.)
Recall that if $\kappa \in (0,8)$, Beffara's theorem (partially proved in \cite{RS} and completed
in \cite{Beffara}) states that the Hausdorff dimension of $SLE_\kappa$ is almost surely $d = 1 + \kappa/8$.

The main purpose of this paper is to remove the ``up to reparametrization'' from the above characterization.
Roughly speaking, we will show that the conformal
invariance assumption \eqref{confinvariance} and the domain Markov property uniquely characterize
the law of the random parametrized path as being an $SLE_\kappa$ with a {\em particular} parametrization
that we will construct in this paper.  We may interpret this parametrization as giving a $d$-dimensional volume measure
on $\gamma$, which is uniquely defined up to a multiplicative constant.  As mentioned in
Section \ref{overviewsubsection}, one major caveat is that, due to limitations
of certain second moment estimates we need, we are currently only able to prove that
this measure is non-trivial (i.e., not identically zero) for $\kappa < 4(7 - \sqrt{33}) = 5.021 \cdots$, although
we expect this to be the case for all $\kappa < 8$.




\section{$SLE$ definition and limit constructions}  \label{definitionsec}

\subsection{Schramm-Loewner evolution ($SLE$)} \label{SLEsubsec}

We now provide a quick review of the definition
of the Schramm-Loewner evolution;
see \cite{LBook}, especially
Chapters 6 and 7, for more details.
  We will discuss only  chordal $SLE$ in
this paper, and we will call it just $SLE$.

Suppose  that $\gamma:(0,\infty) \rightarrow
\Half =\{x+iy: y > 0\}$
is a non-crossing curve with $\gamma(0+) \in
\R$ and $\gamma(t) \rightarrow \infty$ as $t \rightarrow
\infty$.  Let $H_t$ be the unbounded component of
$\Half \setminus \gamma(0,t]$.   Using the Riemann mapping
theorem, one    can see that there is a unique conformal
transformation
\[            g_t: H_t \longrightarrow \Half \]
satisfying $g_t(z) - z \rightarrow 0$ as $z \rightarrow \infty$.
It has an expansion at infinity
\[           g_t(z) = z + \frac{a(t)}{z} + O(|z|^{-2}). \]
The coefficient $a(t)$  equals
$\hcap(\gamma(0,t])$ where $\hcap(A)$ denotes
the half plane capacity from infinity of a bounded set
$A$.  There are a number of ways of defining $\hcap$, e.g.,
\[             \hcap(A) = \lim_{y \rightarrow \infty}
   y\,  \E^{iy}[\Im(B_{\tau})], \]
where $B$ is a complex Brownian motion and $\tau = \inf\{t:
B_t \in \R \cup A\}$.

\begin{definition}
The {\em Schramm-Loewner evolution},  $SLE_\kappa$, (from
$0$ to infinity in $\Half$) is the
 random curve $\gamma(t)$ such that $g_t$
satisfies
\begin{equation}
\label{loewner}
         \dot g_t(z) = \frac{a}{g_t(z) - V_t} ,\;\;\;\;
   g_0(z) = z,
\end{equation}
where $a=2/\kappa$ and $V_t = -B_t$ is a standard Brownian
motion.
\end{definition}

Showing that the
conformal maps $g_t$ are well defined is easy.  In fact,
for given $z \in \Half$, $g_t(z)$ is defined up to
time $T_z = \sup\{t: \Im g_t(z) > 0\}.$   Also,
$g_t$ is the unique conformal transformation of $H_t  =\{
z \in \Half: T_z > t\}$ onto $\Half$ satisfying
$g_t(z) - z  \rightarrow 0$ as $z \rightarrow \infty$.
 It is not as easy to show that $H_t$ is given
by the unbounded component of $\Half \setminus
\gamma(0,t]$ for a curve $\gamma$.  However,
  this was shown
for $\kappa \neq 8$ by Rohde and Schramm
\cite{RS}. If $\kappa \leq 4$, the curve is simple and
$\gamma(0,\infty) \subset \Half$.  If $\kappa > 4$, the
curve has double points and $\gamma(0,\infty ) \cap
\R \neq \emptyset.$  For $\kappa \geq 8,$ $\gamma(0,\infty)$
is plane filling; we will restrict our consideration to
$\kappa < 8$.

\begin{remark}
We have defined chordal $SLE_\kappa$ so that it is {\em
parametrized
by capacity} with \[ \hcap(\gamma(0,t]) = at.\]  It is more often
defined with the capacity parametrization chosen so that
$ \hcap(\gamma[0,t]) = 2t$.  In this case we need to choose
$U_t = - \sqrt \kappa \, B_t$.  We will choose the parametrization
in (\ref{loewner}), but this is only for our convenience.
 Under our parametrization, if $z \in \overline \Half
\setminus \{0\}$, then
 $Z_t = Z_{t}(z)  := g_t(z) - U_t$ satisfies the Bessel equation
\[          dZ_{t} = \frac{a}{Z_{t}}\, dt + dB_t. \]
In this paper we will use both $\kappa$ and $a$ as notations;
throughout, $a = 2/\kappa$. \end{remark}

We let
\[  f_t = g_t^{-1},\;\;\;\; \hat f_t(z) = f_t(z + U_t). \]
We recall the following scaling relation
\cite[Proposition 6.5]{LBook}.

\begin {lemma}[Scaling]   \label{scalinglemma}
  If
$r > 0$, then the distribution of $g_{tr^2}(rz)/r$ is the same
as that of $g_t(z)$; in particular, $g_{tr^2}'(rz)$ has the
same distribution as $g'_t(z)$.
\end{lemma}

  For
$\kappa < 8$,  we let
\begin{equation}  \label{nov28.20}
       d = 1 + \frac \kappa 8 = 1 + \frac 1{4a}.
\end{equation}
 If $z \in \C$ we will write $x_z,y_z$ for the real and
imaginary parts of $z = x_z + iy_z$ and $\theta_z$ for the
argument of $z$.  Let
\begin{equation}  \label{green}
     G(z) := y_z^{d-2}\, [(x_z/y_z)^2 + 1]^{\frac 12 - 2a}
   =|z|^{d-2}  \, \sin^{\frac
  \kappa 8 + \frac 8{\kappa} -2} \theta_z   ,
\end{equation}
denote the ``Green's function'' for $SLE_\kappa$ in
$\Half$.  The value of $d$ and the function $G$ were first found in \cite{RS}
and are characterized by the scaling rule $G(rz) =
r^{d-2} \, G(z)$ and the fact that
\begin{equation}  \label{localmart}
    M_t(z) := |g_t'(z)|^{2-d} \, G(Z_t(z))
\end{equation}
is a local martingale.  In fact, for a given $\kappa$, the scaling rule $G(rz) =
r^{d-2} \, G(z)$ and the requirement that \eqref{localmart} is a local martingale
uniquely determines $d$ and (up to a multiplicative constant) $G$.
Note that
if $K < \infty$,
\begin{equation}  \label{greenintegral}
           \int_{|z| \leq K} G(z) \, dA(z)
  = K^{d} \, \int_{|z| \leq 1} G(z) \, dA(z) < \infty .
\end{equation}
Here, and throughout this paper, we use $dA$ to denote
integration with respect to area.  The Green's function will
turn out to describe the expectation of the measure we intend
to construct in later sections, as suggested by the following
proposition.

\begin{proposition} \label{expectationprop}
Suppose that there exists a parametrization for $SLE_\kappa$ in $\Half$ satisfying the domain Markov
property and the conformal invariance assumption \eqref{confinvariance}.  For a fixed Lebesgue measurable subset $S \subset \Half$,
let $\Theta_t(S)$ denote the process that gives the amount of time in this parametrization
spent in $S$ before time $t$
(in the half-plane capacity parametrization given above), and suppose further that $\Theta_t(S)$ is
$\mathcal F_t$ adapted for all such $S$.  If $\mathbb E \Theta_\infty(D)$ is finite for
all bounded domains $D$, then it must be the case that (up to multiplicative constant)
$$\mathbb E \Theta_\infty(D) = \int_D G(z) dA(z),$$ and more generally, $$\mathbb E[\Theta_\infty(D) - \Theta_t(D)|
\mathcal F_t] = \int_D M_t(z) dA(z).$$
\end{proposition}

\begin{proof}
It is  immediate from the conformal invariance assumption (which in particular implies scale invariance)
that the measure $\nu$ defined by $\nu(\cdot) = \mathbb E \Theta_\infty(\cdot)$ satisfies $\nu(r \cdot) = r^d \nu(\cdot)$ for
each fixed $r > 0$.  To prove the proposition, it is enough to show that $\nu(\cdot) =  \int_\cdot G(z) dA(z)$ (up to a constant factor),
since the conditional statement at the end of the proposition then follows from the domain Markov property
and conformal invariance assumptions.

The first observation to make is that $\nu$ is absolutely continuous with respect to Lebesgue measure,
with smooth Radon-Nikodym derivative.  To see this, suppose that $S$ is bounded away from the real axis,
so that there exists a $t>0$ such that almost surely
no point in $S$ is swallowed before time $t$.  Then the conformal invariance assumption
\eqref{confinvariance} and the domain Markov property imply that
$$\nu (S) = \mathbb E \int_{g_t(S)} |(g_t^{-1})'(z)|^{-d} d\nu(z).$$
The desired smoothness can be then deduced from the fact that the law of the pair $g_t(z), g_t'(z)$
has a smooth Radon-Nikodym derivative that varies smoothly with $z$ (which follows
from the Loewner equation and properties of Brownian motion).
  Recalling the scale invariance, we conclude that $\nu$
has the form $$|z|^{d-2}  \, F(\theta_z) dA(z)$$ for some smooth function $F$.
Standard Ito calculus and the fact that $M_t$ is a local martingale determine $F$ up to a constant factor,
implying that $F(z)|z|^{d-2} = G(z)$ (up to a constant factor).
\end{proof}

\begin{remark}
It is not clear whether it is necessary to assume in the statement of Proposition \ref{expectationprop}
that $\mathbb E \Theta_\infty(D) < \infty$ for bounded domains $D$. It is possible that if one had
$\mathbb E \Theta_\infty(D) = \infty$ for some bounded $D$ then one could use some scaling arguments and the law of large numbers
to show that in fact $\Theta_\infty(D) = \infty$ almost surely for domains $D$ intersected by the path $\gamma$.
If this is the case, then the assumption $\mathbb E \Theta_\infty(D) < \infty$ can be replaced by the
weaker assumption that $\Theta_\infty(D) < \infty$ almost surely.
\end{remark}

\subsection{Attempting to construct the parametrization as a limit} \label{candidatesubsec}
As we mentioned in the introduction, the parametrization we will
construct in Section \ref{natparamsec} is uniquely determined by
certain conformal invariance assumptions and the domain Markov property.
Leaving this fact aside, one could also motivate our definition by noting its similarity
and close relationship to some of the other obvious candidates for a
$d$-dimensional volume measure on an arc of an $SLE_\kappa$ curve.

In this section, we will describe two of the most natural candidates: Minkowski measure and
$d$-variation.  While we are not able to prove that either of these candidates is well
defined, we will point out that both of these candidates have variants
that are more or less equivalent to the measure we will construct in Section \ref{natparamsec}.
In each
case we will define approximate parametrizations
$\tau_n(t)$ and propose that a natural parametrization $\tau$
could be given by
\[
           \tau(t) = \lim_{n \rightarrow \infty}
             \tau_n(t),
\]
if one could show that this limit (in some sense) exists and is non-trivial.

To motivate these constructions, we begin by assuming that any candidate for the natural parametrization should satisfy
the appropriate scaling relationship.  In particular if
$\gamma(t)$ is an $SLE_\kappa$ curve that is parametrized
so that $\hcap[\gamma(0,t]] = at$,  then $\tilde \gamma
(t) = r \gamma(t)$ is an $SLE_\kappa$ curve parametrized
so that  $\hcap[\gamma(0,t]] = r^2at$. If it takes time
$\tau(t)$ to traverse $\gamma(0,t]$ in the natural
parametrization, then it should take time
$r^d \, \tau(t)$ to traverse $\tilde \gamma(0,t]$ in
the natural parametrization.  In particular, it should take
roughly time $O(R^d)$ in the natural parametrization for the path to travel distance $R$.

\subsubsection{Minkowski content}  \label{minksec}

Let
\[    {\mathcal N}_{t,\epsilon}   =
              \{z \in \Half : \dist(z,\gamma(0,t]) \leq \epsilon\}, \]
 \[       \tau_n(t) =  n^{2-d}\, {\rm area}
\, (  {\mathcal N}_{t,1/n} )
   . \]
We call the limit $\tau(t) = \lim_{n \rightarrow \infty} \tau_n(t),$ if it exists, the Minkowski content of
$\gamma(0,t]$.
Using the local martingale \eqref{localmart}
one can show that  as $\epsilon \rightarrow
0+$,
\begin{equation}  \label{jul25.2}
          \Prob\{z \in   {\mathcal N}_{\infty,\epsilon}
   \}  \asymp  G(z) \, \epsilon^{2-d}.
\end{equation}

We remark that a commonly employed alternative to Minkowski content is the
$d$-dimensional Hausdorff content; the Hausdorff content
of a set $X \subset D$ is defined to be the limit as $\epsilon \to 0$
of the infimum --- over all coverings of $X$ by balls with some radii
$\epsilon_1, \epsilon_2, \ldots < \epsilon$ ---
of $\sum \Phi(\epsilon_i)$ where $\Phi(x) = x^d$.  We have at least some intuition, however, to
suggest that the Hausdorff content of $\gamma([0,t])$ will be almost surely zero for all $t$.
Even if this is the case, it may be that the Hausdorff content is non-trivial when $\Phi$ is replaced by another function
(e.g., $\Phi(x) = x^d \log \log x$), in which case we would expect
it to be equivalent, up to constant, to the $d$-dimensional Minkowski measure.

\subsubsection{Conformal Minkowski content}  \label{confminksec}

There is a variant of the Minkowski content that could be
called the {\em conformal Minkowski content}.  Let $g_t$ be the
conformal maps as above.  If $t < T_z$,   let
\[               \Upsilon_t(z) = \frac{\Im[g_t(z)]}
              {|g_t'(z)|}. \]
It is not hard to see that $\Upsilon_t(z)$ is the {\em conformal
radius} with respect to $z$ of the domain $ D(t,z)$,
   the component of $\Half \setminus \gamma(0,t]$
containing $z$.  In other words, if $F: \Disk \rightarrow D(t,z)$
is a conformal transformation with $F(0) = z$, then
$|F'(0)| = \Upsilon_t(z)$.
Using the Schwarz lemma or by doing a
simple calculation, we can see that  $\Upsilon_t(z)$ decreases in
$t$ and hence we can define
\[             \Upsilon_t(z) = \Upsilon_{T_z-}(z), \;\;\;\;
  t \geq T_z. \]
Similarly, $\Upsilon(z) = \Upsilon_\infty(z)$ is well defined; this
is the conformal radius with respect to $z$ of the domain
$D(\infty,z)$.
The Koebe $1/4$-Theorem implies that
$\Upsilon_t(z) \asymp \dist[z,\gamma(0,t] \cup \R]$; in fact, each
side is bounded above by four times the other side.  To prove
 \eqref{jul25.2} one can show that there is a $c_*$ such
that
\[       \Prob\{\Upsilon (z) \leq \epsilon\}
               \sim c_* \, G(z) \, \epsilon^{2-d},
\;\;\;\;  \epsilon \rightarrow 0+. \]
This was first established in \cite{LBook} building on the
argument in \cite{RS}.  The conformal Minkowski content
is defined as in the previous paragraph replacing $
{\mathcal N}_{t,\epsilon} $
with
\[ {\mathcal N}_{t,\epsilon}^*
 =  \{z \in \Half :  \Upsilon_t(z)  \leq \epsilon\}. \]
It is possible that this limit will be easier to establish.
Assuming the limit exists, we can see that the expected amount of
time (using the natural
parametrization)
that $\gamma(0,\infty)$ spends in a bounded domain $D$
should be given (up to multiplicative constant) by
\begin{equation}  \label{jul25.4}
              \int_D G(z) \, dA(z) ,
\end{equation}
where $A$ denotes area.  This formula agrees with Proposition \ref{expectationprop}
and will be the starting point
for our construction of the natural parametrization
in Section \ref{natparamsec}.

\subsubsection{$d$-variation}

The idea that it should take roughly time $R^d$ for the path to move
distance $R$ --- and thus $\tau(t_2)-\tau(t_1)$ should
be approximately $|\gamma(t_2)-\gamma(t_1)|^d$ ---
motivates the following definition.  Let
\[    \tau_n(t)  = \sum_{k=1}^{\lfloor tn
  \rfloor}  \left|\gamma\left(\frac kn \right)
   - \gamma \left(\frac{k-1}{n}\right)\right|^d . \]
More generally, we can consider
\[    \tau_n(t) = \sum_{t_{j-1,n} <t} \left|\gamma(t_{j,n}),
   - \gamma(t_{j-1,n})\right|^d, \]
where $t_{0,n} < t_{1,n} < t_{2,n} < \infty$ is a partition,
depending on $n$, whose mesh goes to zero as $n \rightarrow \infty$,
and as usual $d=1+\kappa/8$.  It is natural
to expect that for a wide class of partitions this limit
exists and is independent of the choice of partitions.
In the case $\kappa = 8/3$, a version of this
was studied numerically
by Kennedy \cite{Kennedy}.

\subsubsection{A variant of $d$-variation}  \label{mysec}
We next propose a variant of the $d$-variation in which
an expression involving derivatives of $\hat f'$ (as defined
in Section \ref{SLEsubsec}) takes the
place of $|\gamma(t_2)-\gamma(t_1)|^d$.
Suppose
$\tau(t)$ were the natural parametrization.
 Since $\tau(1) < \infty$, we would
expect that the average value of
\[               \Delta_n \tau (j) := \tau\left(
   \frac {j+1}n \right) - \tau\left(\frac{j}n\right) \]
would be of order $1/n$ for typical $j \in \{1,2, \ldots, n\}$.  Consider
\[           \gamma^{(j/n)}\left[0,\frac 1n\right]
  := g_{j/n}\left(\gamma\left[\frac jn, \frac{j+1}n\right]
 \right). \]
Since the $\hcap$ of this set is $a/n$, we expect that the
diameter of the set is of order $1/\sqrt n$.  Using
the scaling properties, we guess that the time needed
to traverse $ \gamma^{(j/n)}\left[0,\frac 1n\right]$
in the natural parametrization is of order $n^{-d/2}.$
Using
the scaling properties again, we guess that
\[            \Delta_n \tau (j) \approx
                 n^{-d/2} \, |\hat f'_{j/n}(i/\sqrt n)|^d. \]
This leads us to define
\begin{equation}  \label{apr5.5}
           \tau_n(t) =  \sum_{k=1}^{\lfloor tn
  \rfloor}    n^{-d/2} \, |\hat f'_{k/n}(i/\sqrt n)|^d.
\end{equation}

More generally, we could let
\begin{equation}  \label{oct9.1}
     \tau_n(t) =  \sum_{k=1}^{\lfloor tn
  \rfloor}    n^{-d/2}
  \int_\Half  |\hat f'_{k/n}(z/\sqrt n)|^d\, \nu(dz),
\end{equation}
where $\nu$ is a finite measure on $\Half$.
It will turn out that the
parametrization we construct in Section \ref{natparamsec} can be realized as a limit of this
form with a particular choice of $\nu$.  We expect that (up to a constant factor)
this limit is independent of $\nu$, but we will not prove this.

\section{Natural parametrization} \label{natparamsec}

\subsection{Notation} \label{notationsection}

We now summarize some of the key notations we will use throughout the paper.
For $z \in \Half$, we write
\[   Z_t(z) = X_t(z) + i Y_t(z) = g_t(z) - V_t ,  \]\[
   R_t(z) = \frac{X_t(z)}{Y_t(z)} , \;\;\;\;
      \distsub_t(z) = \frac{Y_t(z)}{|g_t'(z)|},\]
\[   M_t(z) = \distsub_t(z)^{d-2}\, (R_t(z)^2 + 1)
  ^{\frac 12 - 2a} = |g_t'(z)|^{2-d}\, G(Z_t(z)). \]
At times we will write just
$Z_t,X_t,Y_t,R_t,\distsub_t,M_t$ but
it is important to remember that these quantities   depend on $z$.

\subsection{Definition}

We will now give a precise definition of the natural time parametrization.
It will be easier to restrict our attention to the time spent in a fixed domain
bounded away from the real line.
Let $\domain$ denote the set of bounded domains $D \subset
\Half$ with $\dist(\R,D) > 0$.  We write
\[          \domain = \bigcup_{m=1}^\infty
                 \domain_m , \]
where $\domain_m$ denotes the set of domains $D$ with
\[              D \subset \{x+iy : |x| < m , 1/m < y < m\}.\]
 Suppose for the moment that $\Theta_t(D)$ denotes the amount
of time in  the natural parametrization that the curve
spends in the domain $D$.  This is {\em not} defined at
the moment so we are being heuristic.  Using
 \eqref{jul25.4} or Proposition \ref{expectationprop} we expect (up to a multiplicative constant that
we set equal to one)
\[   \E\left[\Theta_\infty(D)\right]
    = \int_D G(z) \, dA(z) . \]
In particular, this expectation is finite for bounded $D$.

Let
$\F_t$ denote the $\sigma$-algebra generated by
$\{V_s: s \leq t\}$.
For any process $\Theta_t$ with finite expectations, we would expect that
\[   \E[\Theta_\infty(D) \mid \F_t] = \Theta_t(D) +
          \E[\Theta_\infty(D) - \Theta_t(D) \mid \F_t]. \]
If $z \in D$, with $t < T_z$, then the Markov property
for $SLE$ can be used to see that the conditional distribution
of $\distsub(z)$ given $\F_t$ is the same as the distribution
of $|g_t'(z)|^{-1}\, \distsub^*$ where $\distsub^*$ is independent
of $\F_t$ with the distribution of $\distsub(Z_t(z))$.
This gives us another heuristic way of deriving the formula in Proposition \ref{expectationprop}:
\begin{eqnarray*}
 \lim_{\delta \rightarrow 0+} \delta^{d-2}
  \, \Prob\{\Upsilon(z)< \delta \mid \F_t\}
 & = &
\lim_{\delta \rightarrow 0+} \delta^{d-2}
  \, \Prob\{ \distsub^* \leq \delta\,  |g_t'(z)| \}\\
  & = & c_* \, |g_t'(z)|^{2-d} \, G(Z_t(z)) = c_* \, M_t(z) .
 \end{eqnarray*}
We therefore see that
\[    \E[\Theta_\infty(D) - \Theta_t(D) \mid \F_t]=
  \Psi_t(D) , \]
where
\[ \Psi_t(D) =      \int_D M_t(z)\,
                 1\{T_z > t\} \, dA(z) . \]

We now use the conclusion of Proposition \ref{expectationprop} to give a precise definition for $\Theta_t(D)$.
The expectation formula from this proposition is
\begin{equation}  \label{psi}
\Psi_t(D)  =
          \E[\Theta_\infty(D) \mid \F_t] - \Theta_t(D).
\end{equation}
The left-hand side is clearly supermartingale in $t$ (since it is a weighted average
of the $M_t(x)$, which are non-negative local martingales and hence supermartingales).  It
is reasonable to expect (though we have not proved this) that $\Psi_t(D)$ is in fact continuous as a function of $D$.
Assuming the conclusion of Proposition \ref{expectationprop}, the first term on the right-hand side is a martingale and the map
$t \mapsto \Theta_t(D)$ is increasing.
The reader may recall the continuous case of the standard
Doob-Meyer theorem \cite{DM}: any
continuous supermartingale can be written uniquely as the sum of a continuous adapted decreasing process
with initial value zero and a continuous local martingale.  If $\Psi_t(D)$ is a continuous supermartingale,
it then follows that \eqref{psi} is its Doob-Meyer decomposition.
Since we have a formula for $\Psi_t(D)$, we could (if we knew $\Psi_t(D)$ was continuous)
simply {\em define} $\Theta_t(D)$ to be the unique continuous,
increasing, adapted process such that
\[
 \Theta_t(D)  +
 \Psi_t(D)
\]
is a local
martingale.

Even when it is not known that $\Psi_t(D)$ is continuous, there is a canonical
Doob-Meyer decomposition that we could use to define $\Theta_t(D)$, although the details are more complicated (see
\cite{DM}).  Rather than focus on these issues,
what we will aim to prove in this paper is that there
exists an adapted continuous decreasing $\Theta_t(D)$ for which $\Theta_t(D) + \Psi_t(D)$ is a {\em martingale}.
If such a process exists, it is obviously unique, since if there were another such
process $\tilde \Theta_t(D)$, then $\Theta_{t}(D) - \tilde \Theta_{t }(D)$ would be a continuous
martingale with paths of bounded variation and hence identically zero.
One consequence of having $\Theta_t(D) + \Psi_t(D)$ be a martingale (as opposed to merely a local martingale) is that $\Theta_t(D)$
is not identically zero; this is because $\Psi_t(D)$ is a strict supermartingale (i.e.,
not a martingale), since it is an average of processes $M_t(x)$ which are strict supermartingales (i.e.,
not martingales).  Another reason for wanting $\Theta_t(D)  +
 \Psi_t(D)$ to be a martingale is that this will imply that $\Theta_t$ (defined below) actually satisfies the hypotheses
Proposition \ref{expectationprop}, and (by Proposition \ref{expectationprop}) is the unique
process that does so.
Showing the existence of an adapted continuous increasing $\Theta_t(D)$ that makes $\Theta_t(D) + \Psi_t(D)$ a martingale takes
work.  We conjecture that this is true for all $\kappa <
8$; in this paper we prove it for
\begin{equation}  \label{kappadef}
 \kappa < \kappa_0 :=  4(7 - \sqrt{33}) = 5.021 \cdots         .
\end{equation}

\begin{definition}$\;$
\begin{itemize}

\item
If $D \in \domain$, then the natural parametrization
$\Theta_t(D)$ is the unique continuous, increasing process such
that
\[         \Psi_t(D) + \Theta_t(D) \]
is a martingale (assuming such a process exists).

\item  If $\Theta_t(D)$ exists for each $D \in \domain$,
we define
\[            \Theta_t = \lim_{m \rightarrow \infty}
    \Theta_t(D_m) , \]
where
$D_m = \{x+iy: |x|<m, 1/m < y < m\}. $
\end{itemize}
\end{definition}

The statement of the main theorem includes a function $\phi$
related to the Loewner flow that is defined
later in \eqref{phidef}.  Roughly speaking, we think
of $\phi$ as
\[   \phi(z) = \Prob\{z \in \gamma(0,1] \mid z \in \gamma(0,
 \infty)\}. \]
This equation as written does not make sense because we
are conditioning on an event of probability zero.  To
be precise it is defined by
\begin{equation}  \label{phidef2}
   \E\left[M_1(z)\right] = M_0(z) \, [1-\phi(z)].
\end{equation}
Note that the conclusion of Proposition \ref{expectationprop} and our definition of $\Theta$ imply
that
$$\mathbb E \Theta_1(D) = \int_D \phi(z) G(z) dA(z).$$  This a point worth highlighting: the hypotheses of
Proposition \ref{expectationprop} determine not only the form of $\mathbb E[\Theta_\infty(D)]$ (up to multiplicative constant)
but also $\mathbb E[\Theta_1(D)]$ and (by scaling) $\mathbb E[\Theta_t(D)]$ for general $t$.

In the theorem below, note that
\eqref{aug14.1} is of the form
 \eqref{oct9.1} where $\nu(dz) = \phi(z) \, G(z) \,
dA(z)$.  Let $\kappa_0$ be as in \eqref{kappadef}
 and let $a_0 = 2/\kappa_0$.  Note that
\begin{equation}  \label{kappa}
\frac {16}{\kappa} + \frac{\kappa}{16} > \frac 72 , \;\;\;\;
  0 < \kappa < \kappa_0.
\end{equation}
We will need this estimate later which puts the
restriction on $\kappa$.

\begin{theorem}  $\;$  \label{maint}

\begin{itemize}

\item For
$\kappa < 8$ that are good in the sense of \eqref{aug9.3}
and all $D \in \domain$,
 there is an adapted,
increasing, continuous process $\Theta_{t}(D)$ with $\Theta_0(D) = 0$
such that
\[     \Psi_{t}(D) + \Theta_{t}(D) \]
is a martingale. Moreover, with probability one for all $t$
\[   \Theta_{t}(D) = \hspace{2in} \]
\begin{equation}  \label{aug14.1}
 \lim_{n \rightarrow \infty}
             \sum_{j \leq t2^n}
      \int_\Half |
 \hat f_{\frac{j-1}{2^{n}}}'(z)|^d \,\phi(z 2^{n/2})\,  G(z)
     \,1\{\hat f_{ \frac{j-1}{2^{n}}}(z) \in D\} \,  dA(z)
,
\end{equation}
where
$\phi$ is defined in \eqref{phidef2}.

\item If $\kappa < \kappa_0$, then
$\kappa$ is good.

\end{itemize}

\end{theorem}

\begin{remark}
The hypotheses and conclusion of Proposition \ref{expectationprop} would imply that the
summands in \eqref{aug14.1} are equal to the conditional expectations $$\mathbb E [\Theta_{j 2^{-n}} (D) - \Theta_{(j-1) 2^{-n}}(D) \mid
\F_{(j-1) 2^{-n}}].$$
\end{remark}

\begin{theorem}  \label{newt2}
For all  $\kappa < 8$ and all $t < \infty$.
\begin{equation} \label{newt}
          \lim_{m \rightarrow \infty}
  \E\left[\Theta_t(D_m)\right] < \infty.
\end{equation}
In particular, if $\kappa <8$ is good, then $\Theta_t$
is a continuous process.
\end{theorem}

\noindent {\bf Sketch of proofs}  The remainder of this paper
is dedicated to proving these theorems. For Theorem \ref{maint},
we start by discretizing time and finding an approximation for
$\Theta_t(D)$.  This is done in Sections \ref{forwardsec}
and \ref{approxsec} and leads to the sum in \eqref{aug14.1}.
This time discretization
is the first step in proving the Doob-Meyer Decomposition
for any  supermartingale.  The difficult step comes in taking
the limit.  For general supermartingales, this is subtle and
one can only take a weak limit, see \cite{Meyer}.  However,
if there are uniform second moment estimates for the approximations,
one can take a limit both in $L^2$ and with probability one.
We state the estimate that we will use in \eqref{aug9.3}, and
we call $\kappa$ good if such an estimate exists.  For
completeness, we give a proof of the convergence in Section
\ref{DMsec} assuming this bound; this section is similar to a
proof of the Doob-Meyer Decomposition for $L^2$ martingales
in \cite{Bass}. H\"older continuity of the paths follows.
The hardest part is proving \eqref{aug9.3} and
this is done in Section \ref{momentsec}.  Two arguments are given:
one  easier proof that works for $\kappa < 4$ and a more
complicated argument that works for $\kappa < \kappa_0$.  We conjecture
that all $\kappa < 8$ are good.   In this section
we also establish \eqref{newt} for all $\kappa < 8$
 (see Theorem \ref{firsttheorem}).  Since $t \mapsto
\Theta_t - \Theta_t(D_m)$ is increasing in $t$, and
$\Theta_t(D_m)$ is continuous in $t$ for good $\kappa$,
the final assertion in Theorem \ref{newt2} follows immediately.

 \medskip

Before proceeding, let
us derive some simple scaling relations.  It is well known
that if $g_t$ are the conformal maps for $SLE_\kappa$ and
$r > 0$, then $\tilde g_t(z) := r^{-1} \, g_{tr^2} (rz)$
has the same distribution as $g_t$.  In fact, it is
the solution of the Loewner equation with driving function
$\tilde V_t = r^{-1} \, V_{r^2t}$.
 The corresponding local martingale
is
\begin{eqnarray*}
 \tilde M_t(z) =
 |\tilde g_{t}'(z)|^{2-d} \, G(\tilde g_t(z) - \tilde V_t)
&  = & | g_{tr^2}'(rz)|^{2-d}
  \, G(r^{-1} Z_t(z)) \\
 & = & r^{2-d} \, M_{r^2t}(rz),
\end{eqnarray*}
\[ \tilde \Psi_t(D) =: \int_D \tilde M_t(z) \, dA(z)
    =  r^{2-d} \int_D  M_{r^2t}(rz) \, dA(z) =
r^{-d } \, \Psi_{r^2t}(rD) . \]
Hence, if $\Psi_t(rD) + \Theta_t(rD)$ is a local martingale, then
so is $\tilde \Psi_t(D) + \tilde \Theta_t(D) , $
where
\[           \tilde \Theta_t(D) = r^{-d} \, \Theta_{r^2t}(rD).\]
This scaling rule implies
that it suffices to prove that $\Theta_t(D)$ exists for $0 \leq t \leq
1$.

\subsection{The forward-time local martingale}    \label{forwardsec}

The process $\Psi_t(D)$ is defined in terms of
the family of local martingales $M_t(z)$ indexed by
starting points $z \in \Half$.  If   $z \not\in \gamma(0,t]$, then $M_t(z)
  $ has a heuristic
interpretation as the (appropriately normalized limit of the)
probability that $z \in \gamma[t,\infty)$
given $\gamma(0,t]$.

Let
$Z_t,X_t,Y_t,R_t,\distsub_t,M_t$ be as defined in Section \ref{notationsection}, recalling
that these quantities implicitly depend on the starting point
$z \in \Half$.
The Loewner equation  can be written as
\begin{equation}  \label{mar1.0}
  dX_t = \frac{aX_t}{X_t^2 + Y_t^2} \, dt + d B_t,
\;\;\;\;\;  \p_tY_t = - \frac{aY_t}{X_t^2 + Y_t^2} ,
\end{equation}
and using It\^o's formula and the chain rule  we see that
if $t < T_z$,
\begin{equation}  \label{mar1.1}
      \p_t \distsub_t = - \distsub_t \, \frac{2aY_t^2}{
   (X_t^2 + Y_t^2)^2},      \;\;\;\;
      d M_t =  M_t \, \frac{(1-4a) \, X_t}
   {X_t^2 + Y_t^2} \, dB_t.
\end{equation}
It is straightforward to
 check that with probability one
\begin{equation}  \label{apr17.2}
      \sup_{0 \leq s < T_z \wedge t} \, M_s(z)
                      \left\{ \begin{array}{ll}= \infty ,& \mbox{ if }
   z \in \gamma(0,t]\\
     <\infty, & \mbox{ otherwise. } \end{array}\right.
\end{equation}
Moreover,
if $4 < \kappa < 8$ and $z \not \in \gamma(0,\infty)$, then
$T_z < \infty$ and
\[
           M_{T_z-}(z) = 0 .
\]
In other words, if we extend $M_t(z)$ to $t \geq T_z$ by
$M_{t}(z) = M_{T_z-}(z),$   then for
$z \not\in \gamma(0,\infty)$, $  M_t(z)$ is continuous in $t$ and
equals zero if $t \geq T_z$.
Since $\gamma(0,\infty)$ has zero area, we can write
\begin{equation}  \label{jul25.6}
 \Psi_t(D) =
 \int_D M_t(z) \, dA(z)   = \int_D M_t(z) \, 1\{T_z > t\}
  \, dA(z).
\end{equation}

\begin{proposition}  If $z \in \Half$, $M_t = M_t(z)$
is a local martingale  but not
a martingale.  In fact,
\begin{equation}  \label{apr5.1}
  \E[M_t] = \E[M_0] \, [1- \phi(z;t)] =G(z)
 \,  [1- \phi(z;t)] .
\end{equation}
Here  $\phi(z;t) = \Prob\{T_z^* \leq t\}$
is the distribution function of
\[  T^*_z = \inf\{t: Y_t = 0 \}, \]
where $X_t+ i Y_t$ satisfies
\begin{equation}  \label{may19.1}
    dX_t = \frac{(1-3a)\, X_t}{X_t^2 + Y_t^2} \, dt
             + dW_t, \;\;\;\;    \p_tY_t = - \frac{a\, Y_t}
  {X_t^2 + Y_t^2},  \;\;\;\; X_0 +i Y_0 = z,
\end{equation}
and $W_t$ is a standard Brownian motion.
\end{proposition}

\begin{proof}  The fact that $M_t$ is a local martingale
follows immediately from
\[              dM_t = \frac{(  1-4a)\, X_t}{X_t^2 + Y_t^2}
  \, M_t \, dB_t. \]
To show that $M_t$ is not a martingale, we will
consider $\E[M_t]$.
For every $n$, let $\tau_n = \inf\{t: M_t \geq n\}$.
Then
\[ \E[M_t] = \lim_{n \rightarrow \infty}
   \E[M_t; \tau_n > t] = \E[M_0] - \lim_{n \rightarrow
 \infty} \E[M_{\tau_n}; \tau_n \leq t ]. \]
If $z \not\in \gamma(0,t]$, then
$M_t(z) < \infty$.
Therefore
\[ \lim_{n \rightarrow
 \infty} \E[M_{\tau_n}; \tau_n \leq t ] \]
denotes the probability that the process $Z_t$ weighted
(in the sense of the Girsanov Theorem)
by $M_t$ reaches zero before time $t$.
We claim that for $t$ sufficiently large,
\begin{equation}  \label{apr17.1}
   \lim_{n \rightarrow
 \infty} \E[M_{\tau_n}; \tau_n \leq t ] > 0 .
\end{equation}
We verify this by using the Girsanov theorem. For fixed $n$,
$M_{t,n}:= M_{t \wedge \tau_n}$ is a nonnegative martingale
satisfying
\[              dM_{t,n} = \frac{(1-4a)\, X_t}{X_t^2 + Y_t^2}
  \, M_{t,n} \, 1\{\tau_n > t\} \,  dB_t. \]
The Girsanov transformation considers the paths under the
new measure $Q = Q^{(n)}$ defined by
\[         Q(E) =  M_{0}^{-1} \, \E[M_{t,n} \, 1_E] \]
if $E$ is ${\cal F}_t$-measurable.   The Girsanov theorem tells us that
in the new measure, $X_t$ satisfies \eqref{may19.1}
where $W_t$ is a standard Brownian motion in the new
measure.    It is fairly straightforward to show that
if $(X_t,Y_t)$ satisfy \eqref{may19.1} and $a > 1/4$,
then $Y_t$ reaches zero in finite time.
\end{proof}

The process satisfying \eqref{may19.1} is called
{\em two-sided radial $SLE_{2/a}$ from  $0$ and $\infty$
to $z$ in $\Half$}.
Actually, it is the distribution only of one of the two arms, the
arm from $0$ to $z$.  Heuristically, we think of this as $SLE_{2/a}$
from $0$ to $\infty$ conditioned so that $z$ is on the path.
Let $T_z^* = \inf\{t: Z_t = 0\}$ where $Z_t = X_t + i Y_t$
satisfies  \eqref{may19.1} with $Z_0 = z.$
  We have noted that $\Prob\{\tau < \infty\}
= 1$.   The function $\phi(z;t)$ will be important.  We define
\begin{equation}  \label{phidef}
    \phi(z) = \phi(z;1) . \end{equation}
 \begin{equation}  \label{apr5.1.plus} \phi_t(z) = \Prob\{t \leq \tau_z \leq t+1\}
  = \phi(z;t+1) - \phi(z;t) .
\end{equation}
In particular, $\phi_0(z) = \phi(z)$.
The scaling properties of $SLE$ imply
\[ \phi(z;t) =\phi(z/\sqrt t),\]

Let $Q = Q_z$ be the probability measure obtained by weighting by
the local martingale $M_t(z)$.  Then $\phi(z;t)$ denotes the
distribution function of $T = T_z$ in the measure $Q$.   If $t,s > 0$, then
\[    Q[t < T < t+ s \mid \F_t]
                    =  \phi(Z_t(z);s) \, 1\{T > t\}. \]
Taking expectations, we get
\begin{equation}  \label{aug13.1}
 \E\left[M_t(z) \, \phi(Z_t(z);s)\right] = G(z) \, [\phi(z;t+s)
   - \phi(z;t)].
\end{equation}
The next lemma describes
 the distribution of $T$ under $Q$
in terms of a functional of a
simple one-dimensional
diffusion.

\begin{lemma}  Suppose $a >1/4$ and
$X_t + iY_t$ satisfies
\[   dX_t = \frac{(1-3a)\,
 X_t}{X_t^2 + Y_t^2} \, dt + dW_t,\;\;\;\;\;
     \p Y_t = - \frac{aY_t}{X_t^2 + Y_t^2} dt,\;\;\;\;
   X_0  = x, \;\; Y_0 =1 , \]
where $W_t$ is a standard Brownian motion.
Let
\[  T = \sup\{t: Y_t > 0\}.\]
Then
\[  T = \int_0^{\infty} e^{-2as} \, \cosh^2 J_s \, ds
   = \frac{1}{4a} + \frac 12 \int_0^{\infty}
 e^{-2as} \, \cosh(2J_s) \, ds , \]
where $J_t$ satisfies
\begin{equation}  \label{aug5.1}
   dJ_t =  (\frac 12-2a) \, \tanh J_t \, dt + dW_t, \;\;\;\;
       \sinh J_0 = x.
\end{equation}
\end{lemma}

\begin{proof}
Define the time change
\[  \sigma(s) = \inf\{t: Y_t = e^{-as} \}. \]
Let $\hat X_s = X_{\sigma(s)}, \hat Y_s
 = Y_{\sigma(s)} = e^{-as}$.  Since
\[   -a \, \hat Y_t= \p_t \hat Y_t = - \dot \sigma(t) \,
\frac {a \, \hat Y_t^2} {\hat X_t^2 + \hat Y_t^2},\]
we have
\[          \dot \sigma(s) = \hat X_t^2 + \hat Y_t^2
   = e^{-2as} \, [K_s^2 + 1] , \]
where $K_s = e^{as} \hat X_s$.
Note that
\[  d\hat X_s = \left(\frac 12 - 3a\right) \, \hat X_s \, ds
             + e^{-as} \, \sqrt{K_s^2 + 1} \, dW_s. \]
\begin{equation}
  \label{aug5.2}
  dK_s =   \left(1-2a\right) \, K_s \, ds +
             \sqrt{K_s^2 + 1} \, dW_s,\;\;\;\; K_0 = x .
\end{equation}
Using It\^o's formula we see that if   $J_s$
satisfies \eqref{aug5.1} and $K_s = \sinh(J_s)$,
then $K_s$ satisfies \eqref{aug5.2}.
Also,
\[ \sigma(\infty) = \int_0^\infty \dot \sigma(s) \, ds
   = \int_0^\infty e^{-2as} \, [K_s^2 +1]\, ds
       = \int_0^\infty e^{-2as} \, \cosh^2 J_s \, ds.\]
\end{proof}

Using the lemma one can readily see that there exist $c,\beta$
such that
\begin{equation}  \label{phiestimate}
   \phi(s(x+iy); s^2) =
\phi(x+iy) \leq c \, 1\{y \leq 2a\} \, e^{-\beta  x^2}.
\end{equation}

\subsection{Approximating $\Theta_t(D)$}  \label{approxsec}

 If $D \in \domain$, the
  change of variables $z = Z_t(w)$ in \eqref{jul25.6}
 gives
\begin{eqnarray}
 \Psi_t(D) &= & \int_\Half
   |g_t'(w)|^{2-d} \, G(Z_t(w)) \, 1\{w \in D\}
  \, dA(w) \nonumber \\
& = & \int_\Half |\hat f_t'(z)|^{d}
  \, G(z) \, 1\{\hat f_t(z) \in D\} \, dA(z) ,  \label{mar1.3}
\end{eqnarray}
\[ \E\left[ \Psi_t(D)\right]
  =
 \int_\Half \E\left[|\hat f_t'(z)|^{d}; \hat f_t(z) \in D\right]
  \, G(z)  \, dA(z).\]

\begin{lemma} If $D \in \domain$, $s,t \geq 0$,
 \begin{equation}  \label{mar1.4}
  \E[  \Psi_{s+t}(D) \mid \F_s] = \Psi_s(D) -
               \int_\Half
 |\hat f_s'(z)|^d  \, G(z) \, \phi(z/\sqrt t)\, 1\{\hat f_s(z) \in D\} \, dA(z)
,\end{equation}
where $\phi(w)$ is as defined in \eqref{apr5.1} and \eqref{apr5.1.plus}.
\end{lemma}

\begin{proof}  Recalling the
definition of  $\phi(w;t)$  in \eqref{apr5.1},
we get
\begin{eqnarray*}
\lefteqn{\E[  \Psi_{s+t}(D) \mid \F_s]}  \\
 &  = &
  \E \left [ \int_D  M_{s+t}(w) \, dA(z)
\left|\right.
  \F_s \right]  \\
    & = &  \int_D  \E[ M_{s+t}(w) \mid \F_s]
                  \, dA(w)  \\
   & = & \int_D  M_s(w) \,[1 - \phi(Z_s(w);t) ]\, dA(w)
  \\
  & = & \Psi_s(D)-\int_D  M_s(w) \,  \phi(Z_s(w);t)   \, dA(w), \\
& = & \Psi_s(D)-\int_\Half  |g_s'(w)|^{2-d}
\, G(Z_s(w)) \,  \phi(Z_s(w);t)  \, 1\{w \in D\} \, dA(w).
\end{eqnarray*}
If we use the change of variables  $z = Z_s(w)$ and the scaling
rule for $\phi$, we get
\eqref{mar1.4}.\end{proof}

Using the last lemma, we see that a
 natural candidate for the process $\Theta_t(D)$ is given
by
\[ \Theta_t(D) = \lim_{n \rightarrow \infty} \Theta_{t,n}(D),\]
where
\begin{eqnarray*}
  \Theta_{t,n}(D) & = &
 \sum_{j \leq t2^n} \E[\Psi_{j2^{-n}}(D) - \Psi_{(j-1)
  2^{-n}}(D) \mid \F_{(j-1)2^{-n}}]\\
  & = & \lim_{n \rightarrow \infty}
  \sum_{j \leq t2^n} I_{j,n} (D)  ,
\end{eqnarray*}
and
  \begin{equation}\label{aug9.1}
    I_{j,n} (D) =
     \int_\Half |
 \hat f_{(j-1)2^{-n}}'(z)|^d \,\phi(z 2^{n/2})\,  G(z)
   \, 1\{\hat f_{(j-1)2^{-n}}(z) \in D \} \, dA(z) .
\end{equation}
Indeed, it is immediate that for fixed $n$
\[    \Psi_{t,n}(D) + \Theta_{t,n}(D) \]
restricted to $t = \{k2^{-n}: k=0,1,\ldots\}$ is a
martingale.

To take the limit, one needs further conditions. One
sufficient condition (see Section \ref{DMsec}) is
a second moment condition.

\begin{definition}  We say that $\kappa $ (or $a =2/\kappa$)
is {\em good} with exponent $\alpha > 0$
  if $\kappa < 8$ and
 the following holds.
For every  $D \in
 \domain$
there exist
$c < \infty$ such that for all $n$,  and
 all $s, t \in \dyadic_n$ with $0 < s<t \leq 1$,
\begin{equation}  \label{aug9.3}
 \E\left[[\Theta_{t,n}(D) - \Theta_{s,n}(D)]^2\right]
          \leq c  \, (t-s)^{1 + \alpha}
   .
\end{equation} We say $\kappa$ is {\em good} if this holds for some $\alpha > 0$
\end{definition}

By scaling we see that
$I_{j,n}(D)$ has the same distribution as
\[        \int_\Half |
 \hat f_{j-1}'(z2^{n/2} )|^d \,\phi(z  2^{n/2} )\,  G(z)
   \, 1\{\hat f_{j-1}(z2^{n/2}) \in 2^{n/2}\, D \} \, dA(z)
  , \]
and the change of variables $w = z2^{n/2}$ converts this integral to
\[       2^{-nd/2}  \int_\Half |
 \hat f_{j-1}'(w)|^d
   \, 1\{\hat f_{j-1}(w) \in 2^{n/2} \, D\} \, d\mu(w),\]
where $d\mu(w) = \phi(w) \, G(w) \, dA(w)$.
Therefore,
\[  \Theta_{t,n}(D) := \sum_{
j \leq t 2^n}  I_{j,n}(D) \]
has the same distribution as
\begin{equation}  \label{jan17.1}
     2^{-nd/2}   \sum_{j \leq t2^{n}}
  \int_\Half |
 \hat f_{j-1}'(w)|^d
   \, 1\{\hat f_{j-1}(w) \in 2^{n/2} \, D\} \, d\mu(w).
\end{equation}

\section{Doob-Meyer decomposition} \label{DMsec}

In this section, we give a  proof of the Doob-Meyer
decomposition for submartingales satisfying
a second
moment bound.  Although we will apply the results
in this section to $  \Theta_{t}(D)$
with $D \in \domain$, it will be easier to abstract the
argument and write just $L_t=-\Theta_t$.
   Suppose $L_t$ is a submartingale
with respect to $\F_t$ with $L_0 = 0$.
We call
a finite subset of $[0,1]$,
  $\dyadic$,
which contains $\{0,1\}$ a {\em partition}.
We can write the elements of a partition as
\[     0= r_0 < r_1 < r_2 < r_k = 1 . \]
We define
\[  \|\dyadic\| = \max \{r_j - r_{j-1}\} , \;\;\;\;\;
    c_*(\dyadic) =  \|\dyadic\|^{-1}  \min
\{r_j - r_{j-1} \}. \]

\begin{definition}  Suppose $\dyadic_n$ is a sequence
of partitions.
\begin{itemize}
\item If $\dyadic_1 \subset \dyadic_2 \subset
\cdots,$ we call the sequence {\em increasing} and let
$\dyadic = \cup \dyadic_n$.
\item  If there exist $0 < u,c<\infty$ such that for each
$n$, $\|\dyadic_n \| \leq c e^{-un}$ we call the
sequence  {\em geometric}.
\item  If there exists $c > 0 $ such that for all $n$,
$c_*(\dyadic_n) \geq c$, we call the sequence
{\em regular}.
\end{itemize}
\end{definition}

The prototypical example of an increasing, regular, geometric
sequence of partitions is the dyadic rationals
\[  \dyadic_n = \{k2^{-n}: k=0,1,\ldots,2^n\}. \]

Suppose $L_t, 0 \leq t \leq 1$
 is a submartingale with respect to
$\F_t$.  Given a sequence of partitions $\dyadic_n$ there exist
increasing
processes  $\Theta_{r,n}, r \in \dyadic_n$ such that
\[            L_r - \Theta_{r,n} , \;\;\;\;
    r \in \dyadic_n \]
is a martingale.  Indeed if $\dyadic_n$ is given by
$0 \leq r_0 < r_1 <  \ldots < r_{k_n}=1$, then we can define
the increasing process by
$\Theta_{0,n} = 0$ and recursively
\[    \Theta_{r_j,n} = \Theta_{r_{j-1},n} +
    \E\left[ L_{r_j} - L_{r_{j-1}}  \mid \F_{r_{j-1}}\right]
.\]
Note $ \Theta_{r_j,n} $ is $\F_{r_{j-1}}$-measurable
and if $s,t \in \dyadic_n$ with $s < t$,
\begin{equation}  \label{apple.11}
\E[\Theta_{t,n} \mid \F_s] =
             \Theta_{s,n} + \E[ L_t - L_s \mid \F_s] .
\end{equation}
The proof of the next proposition follows the proof of
the Doob-Meyer Theorem in \cite{Bass}.

\begin{proposition}  Suppose $\dyadic_n =\{r(j,n);j=0,1,\ldots,k_n\}
$ is an increasing
sequence of partitions with
\[  \delta_n := \|\dyadic_n\| \rightarrow 0. \]
Suppose there exist
$\beta > 0$ and $c < \infty
 $ such that for all $n$ and all  $s,t \in
\dyadic_n$
\begin{equation} \label{momentcond}
    \E\left[(\Theta_{t,n} - \Theta_{s,n})^2\right]
    \leq  c\, (t-s)^{\beta + 1}.
\end{equation}
Then there exists an increasing, continuous process
$\Theta_t$ such that $L_t - \Theta_t$ is a martingale.
Moreover, for each $t \in \dyadic$,
\begin{equation}  \label{apple.53}
            \Theta_t = \lim_{n \rightarrow \infty} \Theta_{t,n},
\end{equation}
where the limit is in $L^2$.  In particular, for all $s< t$,
\[ \E\left[(\Theta_{t} - \Theta_{s})^2\right]
    \leq   c\, (t-s)^{\beta + 1}. \]
If $u < \beta/2$, then with probability one, $\Theta_t$ is
H\"older continuous of order $u$.
If the sequence is geometric (i.e., if $\delta_n \to 0$ exponentially in $n$), then the limit in
\eqref{apple.53} exists with probability one.
\end{proposition}

\begin{proof}  We first consider the limit for $t=1$.
  For $m \geq n$,
let
\[  \Delta(n,m) = \max\left\{\Theta_{r(j,n),m}-\Theta_{r(j-1,n),m}:
   j=1,\ldots,k_n\right\}. \]
Using \eqref{momentcond} and writing $k=k_n$,
 we can see that
\begin{eqnarray*}
\E\left[\Delta(n,m)^2\right] & \leq & \sum_{j=1}^k
  \E\left[(\Theta_{r(j,n),m} - \Theta_{r(j-1,n),m})^2\right]\\
       & \leq
  &         c \sum_{j=1}^k [r(j,n)- r(j-1,n)]^{\beta + 1} \,
  \\
 & \leq & c\, \delta_n^\beta \, \sum_{j=1}^k [r(j,n)- r(j-1,n)] \,
   \\
 & = & c \, \delta_n^\beta .
\end{eqnarray*}
If $m \geq n$, then \eqref{apple.11} shows
that
$Y_t:=
\Theta_{t,m} - \Theta_{t,n}, t \in \dyadic_n$
is  a martingale, and \eqref{momentcond} shows that it is
square-integrable.  Hence, with $k = k_n$,
\begin{eqnarray*}
\E\left[(\Theta_{1,m} - \Theta_{1,n})^2 \right]
  & = & \sum_{j=1}^k \E\left[(Y_{r(j,n)} - Y_{r(j-1,n)})^2
  \right]  \\
   & \leq & \E\left[(\Delta(n,n) + \Delta(n,m))
         \sum_{j=1}^{k}  \left|Y_{r(j,n)} - Y_{r(j-1,n)}
\right| \right]
\end{eqnarray*}
Note that
\begin{eqnarray*}
\lefteqn{ \sum_{j=1}^{k}  |Y_{r(j,n)} - Y_{r(j-1,n)}| }\\&
\leq & \sum_{j=1}^k \left([\Theta_{r(j,n),n} - \Theta_{r(j-1,n),n}]
   + [ \Theta_{r(j,n),m} - \Theta_{r(j-1,n),m}] \right)\\
 & = &\Theta_{1,n} + \Theta_{1,m}.
\end{eqnarray*}
Therefore,
\begin{eqnarray}
 \lefteqn{\E\left[(\Theta_{1,m} - \Theta_{1,n})^2 \right]}
\hspace{.5in}\nonumber\\
 & \leq & \E\left[(\Delta(n,n) + \Delta(n,m))
                   \,(\Theta_{1,n} + \Theta_{1,m})\right] \nonumber \\
  & \leq & \E[(\Delta(n,n) + \Delta(n,m))^2]^{1/2} \;
   \E[(\Theta_{1,n} + \Theta_{1,m})^2]^{1/2} \nonumber \\
  & \leq & c \,\delta_n^\beta \label{apple.59}
\end{eqnarray}

 This shows that $\{\Theta_{1,m}\}$
 is a  Cauchy sequence, and
completeness of $L^2$ implies  that there is a limit.
Similarly, for every $t \in \dyadic$, we can see that
there exists
 $\Theta_t$ such that
\[          \lim_{n \rightarrow \infty}
     \E\left[|\Theta_{t,n} - \Theta_t|^2\right]
  =0 ,\;\;\;\;  t \in \dyadic. \]
Moreover, we get the estimate
\[           \E[(\Theta_{t} - \Theta_s)^2]
 = \lim_{n \rightarrow \infty} \E[ (\Theta_{t,n} - \Theta_{s,n})^2]
     \leq c \, (t-s) ^{\beta +1}\]
for
\[
 0 \leq s \leq t  \leq 1, \;\;\; s,t \in \dyadic. \]
The $L^2$-maximal inequality implies then that
\[    \E\left[\sup_{s \leq r \leq t} (\Theta_{r} - \Theta_s)^2
 \right] \leq c \,  (t-s) ^{\beta +1}
 \;\;\;\; 0 \leq s \leq t  \leq 1, \;\;\; s,t \in \dyadic,\]
where the supremum is also restricted to $r \in \dyadic$.
Let
\[   M(j,n) = \sup\left\{(\Theta_t - \Theta_s)^2:
      (j-1)2^{-n} \leq s,t \leq j2^{-n} , s,t \in \dyadic
  \right\}, \]
\[  M_n = \max\{M(j,n):j=1,\ldots,n\}.\]
Since $\dyadic$ is dense, we can then conclude
\[  \E[M(j,n)] \leq c \, 2^{-n(\beta+1)} , \]
\[  \E[M_n] \leq \sum_{j=1}^{2^n} \E[M(j,n)] \leq c\, 2^{-n\beta}. \]
An application of the triangle inequality shows that
if
\[   Z_n = \sup\left\{(\Theta_{t} - \Theta_s)^2: 0 \leq s,t \leq 1,
s,t \in \dyadic, |s-t|\leq 2^{-n} \right\}, \]
then
\[  \E[Z_n] \leq  c\, 2^{-n\beta}. \]
 The Chebyshev inequality and
the Borel-Cantelli Lemma show that if  $u < \beta/2$, with
probability one
\[  \sup\left\{\frac{|\Theta_{t} - \Theta_s|}{(t-s)^u} :
  0 \leq s , t \leq 1 , s,t \in \dyadic \right\} < \infty . \]
In particular, we can choose a continuous version
of the process $\Theta_t$ whose
 paths are H\"older continuous of order $u$ for
every $u< \beta/2$.

If the sequence is geometric, then \eqref{apple.59} implies that
there exist $c,v$ such that
\[     \E\left[(\Theta_{1,n+1} - \Theta_{1,n})^2 \right]
  \leq c \, e^{-nv}, \]
which implies
\[  \Prob\{|\Theta_{1,n+1} - \Theta_{1,n}| \geq e^{-nv/4}\}
  \leq c \, e^{-nv/2} . \]
Hence by the Borel-Cantelli Lemma we can write
\[    \Theta_{1} =  \Theta_{1,1} + \sum_{n=1}^\infty
             [\Theta_{n+1,1} - \Theta_{n,1}], \]
where the sum converges absolutely with probability one.

\end{proof}

\section{Moment bounds}   \label{momentsec}

In this section we show that $\kappa$ is good for
$\kappa < 4$.  Much of what we do applies to other
values of $\kappa$, so for now we let $\kappa < 8$.
Let
\[   d\mu(z) = G(z) \, \phi(z) \, dA(z), \;\;\;\;
  d \mu_t(z) = G(z) \, \phi(z;t) \, dA(z).\]
We note the scaling rule
\[   d\mu_t(z) = t^{d/2} \, d\mu(z/\sqrt t).\]
>From \eqref{phiestimate} we can see that
\begin{equation} \label{phi2}
d\mu_{t^2}(x+iy)
 \leq c \, y^{d-2} \, [(x/y)^2 + 1]^{\frac 12-2a}
  \, e^{-\beta(x/t)^2} \, 1\{y \leq 2at\}\, dx \, dy.
\end{equation}
Note that this implies (with a different $c$)
\begin{equation}  \label{phi3} d\mu_{t^2}(z)
 \leq c \, [\sin \theta_z]^{\frac \kappa 8 + \frac
 8 \kappa - 2} \, |z|^{\frac  \kappa 8-1}
  \, e^{-\beta|z|^2/t^2} \, dA(z).
\end{equation}

We have shown that $\Theta_{t,n}(D)$ has the same
distribution as $\newtheta_{t,2^{n/2}}(D)$ where

 \[ \newtheta_{t,n}(D)
  = n^{-d} \sum_{j \leq tn^2}
  I_{j-1,nD}\;\; ,\]
and
\[  I_{s,D} =
    \int_\Half |\hat f_s'(w)|^d \, 1\{f_s(w)
  \in D \} \, d\mu(w) . \]
 In this section  we establish the following
theorems which are the main
estimate.

\begin{theorem}  \label{firsttheorem}
If $\kappa < 8$, there exists $c$ such that for
all $s$,
\begin{equation}  \label{firstmoment}
  \E[I_{s,\Half}] \leq c \, s^{d-2}.
\end{equation}
\end{theorem}

\begin{theorem} \label{mainestimate}
  If $\kappa <4$, then for every $m < \infty$
there exists $c = c_m$ such that if
$D \in \domain_m$ and  $1 \leq s,  t \leq n$, then
\begin{equation}  \label{aug10.11}
  \sum_{j=0}^{s^2-1}
  \E[I_{j + t^2,nD} \, I_{t^2, nD}]
            \leq c \, (s/t)^{\zeta}  \, s^{2(d-1)}, \;\;\;\; s\leq t,
\end{equation}
 where $\zeta = 2 - \frac {3 \kappa}4$.  In particular,
\begin{eqnarray*}
 \E \left[[\Theta_{1,n}(D) - \Theta_{\delta,n}(D)]^2
\right] & = &   2^{-nd}
   \sum_{(1-\delta)2^{n} \leq j,k \leq 2^{n}}
  \E[I_{j ,2^{n/2}D} \, I_{k, 2^{n/2} D}]\\
           &  \leq & c \, (1-\delta)^{d + \frac
  \zeta 2} = c \, (1-\delta)^{2- \frac \kappa 4}.
\end{eqnarray*}
\end{theorem}

\begin{theorem} \label{mainestimate2}
  If $\kappa < \kappa_0$,  then for every $m < \infty$
there exists $c = c_m$ such that if
$D \in \domain_m$ and  $1 \leq s,  t \leq n$, then
\begin{equation}  \label{aug10.11.2}
  \sum_{j=0}^{s^2-1}
  \E[I_{j + t^2,nD} \, I_{t^2, nD}]
            \leq c \, (s/t)^{\zeta}  \, s^{2(d-1)}, \;\;\;\; s\leq t,
\end{equation}
 where $\zeta = \frac 4 \kappa - \frac{3 \kappa}
  {16} - 1$.  In particular,
\begin{eqnarray*}
 \E \left[[\Theta_{1,n}(D) - \Theta_{\delta,n}(D)]^2
\right] & = &   2^{-nd}
   \sum_{(1-\delta)2^{n} \leq j,k \leq 2^{n}}
  \E[I_{j ,2^{n/2}D} \, I_{k, 2^{n/2} D}]\\
           &  \leq & c \, (1-\delta)^{d + \frac
  \zeta 2} = c \, (1-\delta)^{\frac 12
 + \frac 2 \kappa + \frac \kappa
{32}}.
\end{eqnarray*}
\end{theorem}

This section is devoted to proving
Theorems \ref{firsttheorem}---\ref{mainestimate2}
Note that
\[  I_{s +t,D} \, I_{t,D}
       =
 \int_\Half \int_\Half |\hat f_{s+t}'(z)|^d
  \,  |\hat f_t'(w)|^d \, 1\{\hat f_{t+s}(z),
 \hat f_t(w)  \in D
               \} \, d\mu(z) \, d\mu(w). \]
In particular,
\begin{equation}  \label{jul25.1}
 \E\left[ I_{s,D}\right]  =
    \int_\Half \E\left[|\hat f_s'(w)|^d ; \hat f_s(w)
  \in D       \right] \, d\mu(w) ,
\end{equation}
\[  \E\left[ I_{s +t,D} \, I_{t,D}\right]
      =  \hspace{2.5in}  \]
\begin{equation}  \label{jul25.2.1}
 \int_\Half \int_\Half \E\left[ |\hat f_{s+t}'(z)|^d
  \,  |\hat f_t'(w)|^d ; \hat f_{t+s}(z) ,
\hat f_t(w) \in D
               \right] \, d\mu(z) \, d\mu(w).
\end{equation}

\subsection{Reverse-time flow}  \label{reversesec}

We will use
  the reverse-time flow for the
Loewner equation to estimates moments
of $\hat f'$. In this subsection we define the reverse
flow and set up some notation that will be useful.
Suppose $s,t \geq 0$ are given.  The expectations we
need to estimate are of the form
\begin{equation}  \label{jul17.1}
\E \left[|\hat f_{t}'(z)|^d
 \right],
\end{equation}
\begin{equation}  \label{jul17.2}
 \E \left[|\hat f_{s+t}'(z)|^d\, |\hat f_{t}'(w)|^d;  \hat f_{s
 + t} (z) \in D, \,  \hat f_{t}(w)  \in D\right],
\end{equation}
We fix $s,t \geq 0$ and allow quantities in this subsection
to depend implicitly on $s,t$.

Let $\tilde U_{r} = V_{t+s - r} - V_{s + t}$.  Then $
\tilde B_{r} := -
\tilde U_{r}, 0 \leq r \leq s + t$ is a standard Brownian motion
starting at the origin.
 Let $U_{r} =  V_{t - r} - V_{t} =
 \tilde U_{s + r} -
\tilde U_{s} , 0 \leq r \leq t$. Then $ B_{r}
= -  U_{r}$ is also a standard Brownian motion and
$\{\tilde U_r: 0 \leq r \leq s\}$ is independent of
$\{ U_r: 0 \leq r \leq t\}$.

Let $\tilde h_{r}, 0 \leq r \leq s + t $ be the
 solution to
the reverse-time Loewner equation
\begin{equation}  \label{reverse}
\p_r \tilde h_r(z) = \frac{a}{\tilde U_r -
  \tilde h_r(z)}, \;\;\;\;
                 \tilde h_0(z) = z .
\end{equation}
Let $ h_r,  0 \leq r \leq t$ be the solution
to
\[\p_r  h_r(z) = \frac{a}{U_r -  h_r(z)}
   = \frac{a}{\tilde U_{s+r} - [h_r(z) + \tilde U_s]}, \;\;\;\;
                 h_0(z) = z . \]
Let $\tilde h = \tilde h_{s+t},
  h = h_{t}$.  Using
only the Loewner equation, we can see that
\[       \hat f_{s+t}(z) =  \tilde h_{s+t}(z) - \tilde U_{s+t}, \;\;\;\;
       \hat f_{t}(w) =  h_t(w) -  U_{t}
  ,\]
\begin{equation}  \label{grape.2}
    \tilde h
  _{s+t}(z)   = h_t( \tilde h_s(z) - \tilde U_s) + \tilde U_s.
\end{equation}
Therefore the expectations in \eqref{jul17.1} and \eqref{jul17.2}
equal
\begin{equation}  \label{jul17.3}
     \E\left[| h'(z)|^d
\right],
\end{equation}
\begin{equation}  \label{jul17.4}
\E\left[ |\tilde h'(z)|^d \, | h'(w)|^d; \tilde h(z) -
 \tilde U_{s + t} \in D,
   h(w) -  U_{t}  \in D \right],
\end{equation}
respectively.  Let
\[   \I_t(z;D) = 1\{h_t(z) - U_t \in D\}, \;\;\;\;
    \I_t(z,w;D) = \I_t(z;D) \, \I_t(w;D).\]
Using \eqref{grape.2},  we can   write \eqref{jul17.4} as
\begin{equation}  \label{jul17.4.alt}
 \E\left[ |\tilde h_s'(z)|^d \,
   | h'_t(h_s(z) - \tilde U_s)|^d \, |h'_t(w)|^d;
  \I_t(\tilde h_s(z) - \tilde U_s, w;D) \right].
\end{equation}
We will derive estimates for $h,\tilde h$.    Let $F_D(z,w;s+t,t)$
denote the expectation in \eqref{jul17.4} and let
\[
  F_D(z,w,t) = F_D (z,w;t,
t) = \E\left[ |
h'_t(z)|^d \, |h'_t(w)|^d\; \I_t(z,w;D) \right].\]
We note the scaling relation: if $r > 0$,
\[  F_{D}(z,w,s+t,t) =
   F_{rD}(rz,rw;r^2(s+t),r^2t).\]
Since $\tilde h_s$ and $h_t$ are independent, we can
see by
 conditioning on the $\sigma$-algebra generated
by $\{\tilde U_r: 0 \leq r \leq s\}$ we see that \eqref{jul17.4.alt}
yields
\begin{equation}  \label{imp1}
  F(z,w;s+t,t) =
   \E\left[|h_s'(z)|^d \,F_D(h_s(z) - U_s,w,t) \right]
\end{equation}
(since $\tilde h_s$ and $h_s$ have the same distribution, we
replaced $\tilde h_s,\tilde U_s$ with $h_s,U_s$).

We rewrite \eqref{jul25.1} and \eqref{jul25.2.1} as
\[   \E\left[I_{t,D}\right] = \int_\Half \E\left[|h_t'(w)|^d
\; \I_t(w;D)
\right] \, d\mu(w), \]
\begin{equation}  \label{jan17.8}
 \E\left[I_{s + t,D} \, I_{t,D}\right]
   = \int_\Half \int_\Half F_D(z,w,s + t,t)
  \, d\mu(w) \, d\mu(z).
\end{equation}

The expressions on the left-hand side
of  \eqref{aug10.11} and
\eqref{aug10.11.2} involve expectations at two
different times.  The next lemma shows that we can
write these sums in terms of ``two-point''
estimates at a single time.
Recall the definition of $\mu_{s}$ from
Section \ref{forwardsec}.

\begin{lemma}  \label{aug11.lemma4} For all $D \in
\domain$ and  $s \geq 0$,
\[    \int_\Half \int_\Half F_D(z,w,s + t,t)
  \, d\mu(w) \, d\mu(z)
=  \hspace{1in} \]
\[ \hspace{1in} \int_\Half \int_\Half
 F_D(z,w,t)
  \, d\mu(w) \, [d\mu_{s + 1} - d \mu_{s}](z).\]
In particular,
  if $s$ is an integer,
\[
 \sum_{j=0}^{s^2-1}  \E[I_{j+t,D} \, I_{t,D}]
=    \int_\Half \int_\Half
 F_{D}(z,w,t)
  \, d\mu(w) \, d\mu_{s^2} (z) .\]
\end{lemma}

\begin{proof}

Using \eqref{imp1}, we write
\[    \int_\Half \int_\Half F_D(z,w,s + t,t)
  \, d\mu(w) \, d\mu(z)
  =
 \int_\Half \E [\Phi]
 \, d\mu(w),\]
where
\[  \Phi = \Phi_D(w,s,t) = \int_\Half  | h_s'(z)|^d \,
     F_D(h_s(z) - U_s,w,t) \,\phi(z) \, G(z)
 \, dA(z).\]
We will change   variables,
\[       z' =  h_{s}(z) - U_{s} = \hat f_{s}(z)  . \]
Here  $ h_{s}(z) = \hat f_{s}(z) + U_{s}  =
 g_{s}^{-1}(z+ U_{s})+ U_{s}$ for a conformal map $g_{s}$ with
the distribution of the forward-time flow with driving function
$\hat U_r = U_{s-r} - U_{s}$. Then
\[    z = \hat f_{s}^{-1}(z') =  g_{s}(z') - \hat U_{s}  =
   \hat Z_{s}(z'), \]
where $ \hat Z_s(\cdot) = g_s(\cdot) -\hat U_s$.
Then
\begin{eqnarray*}
  \Phi & = & \int_\Half  | g_{s}'(z')|^{2-d} \,
     F(z',w,t) \, G(\hat Z_{s}(z')) \, \phi(\hat Z_{s}(z'))  \, dA(z')\\
 & = & \int_\Half \hat M_{s}(z') \,   F(z',w,t) \,  \phi(\hat Z_{s}(z'))  \, dA(z'),
\end{eqnarray*}
 where $\hat M_{s}$ denotes the forward direction local martingale
as in Section \ref{forwardsec}.
Taking expectation using \eqref{aug13.1}, we get
\begin{eqnarray*}
   \E[\Phi] & = & \int_\Half F(z',w,t) \, G(z') \, [\phi(z';s+1)
  - \phi(z';s)] \, dA(z') \\
 & = & \int_\Half F(z',w,t) \,   d[\mu_{s + 1} -
  \mu_{s}](z').
\end{eqnarray*}
This gives the first assertion.  The second assertion follows
from \eqref{jan17.8}.
\end{proof}

\subsection{The reverse-time martingale}

In this section we will collect facts about
the reverse-time martingale; this is analyzed in
more detail in \cite{Lmulti}.
Suppose that $B_t$ is a standard Brownian motion
and $h_t(z)$ is the solution to the reverse-time
Loewner equation
\begin{equation}  \label{revloew}
       \p_t h_t(z) = \frac{a}{U_t - h_t(z)}, \;\;\;\;
   h_0(z) = z,
\end{equation}
where $U_t = - B_t$ and $a = 2/ \kappa$.
Here $z \in \C \setminus \{0\}$.  The solution exists
for all times $t$ if $z \not\in \R$ and
\[    \overline{h_t(z)} = h_t(\overline z).\]
For fixed $t$, $h_t$ is a conformal transformation
of $\Half$ onto a subdomain of $\Half$.
Let
\[   Z_t(z) = X_t(z) + i Y_t(z) =
           h_t(z) - U_t , \]
and note that
\[   dZ_t(z) = -\frac a {Z_t(z)} \, dt + dB_t, \]
\[   dX_t(z) = - \frac{X_t(z)}{|Z_t(z)|^2} \, dt
     + dB_t, \;\;\;\;\;  \p_t Y_t(z) = \frac{a}
    {|Z_t(z)|^2}.\]
We  use $d$ for stochastic differentials and $\p_t$
for actual derivatives.
Differentiation of \eqref{revloew} yields
\[   \p_t |h_t'(z)| = |h_t'(z)| \, \frac{a \,[X_t(z)^2
   - Y_t(z)^2]} {|Z_t(z)|^4} , \]
\[ \p_t \, \left[\frac{|h_t'(z)|}{
  Y_t(z)} \right] = - \left[\frac{|h_t'(z)|}{
  Y_t(z)} \right] \,\frac{2a\, Y_t(z)^2}{|Z_t(z)|^4} .\]
In particular,
\begin{equation}  \label{devbound}
 |h_t'(x+i)| \leq Y_t(x+i) \leq \sqrt{2at + 1} .
\end{equation}

Let
\[
      N_t(z) = |h_t'(z)|^{d} \, Y_t(z)^{1-d} \, |Z_t(z)|
   = |h_t'(z)|^{d} \, Y_t(z)^{-\frac \kappa 8} \, |Z_t(z)|.
\]
An It\^o's formula calculation shows that
 $N_t(z)$ is a martingale satisfying
\[
  dN_t(z) = \frac{ X_t(z)}{|Z_t(z)|^2}
  \, N_t(z) \, dB_t.
\]
More generally, if
    $r  > 0$ and
\[   \lambda = r \, \left[ 1 + \frac \kappa 4
  \right]  \, - \frac{\kappa r^2}{8} , \]
then
\begin{equation}  \label{martingale}
  N_t = N_{t,r}(z) = |h_t'(z)|^\lambda \, Y_t^{- \frac{\kappa r^2}{8}}
  \, |Z_t(z)|^{r} ,
\end{equation}
is a martingale satisfying
 \begin{equation}  \label{jan15.1.alt}
  dN_t(z) = \frac{r\, X_t(z)}{|Z_t(z)|^2}
  \, N_t(z) \, dB_t.
\end{equation}

Note that
\[    h_t(z) - h_t(w) = Z_t(z) - Z_t(w) , \]
\[ \p_t[Z_t(z) - Z_t(w)] = [Z_t(z) - Z_t(w)] \, \frac{a}{
                     Z_t(z) \, Z_t(w)}, \]
\begin{eqnarray*}
 \p_t |Z_t(z) - Z_t(w)| & = & |Z_t(z) - Z_t(w)|\,
  \Re\left[\frac{a}{
                     Z_t(z) \, Z_t(w)}\right]\\
  & = &  |Z_t(z) - Z_t(w)|\, \frac{a[X_t(z) \, X_t(w)
   - Y_t(z) \, Y_t(w)]}{|Z_t(z)|^2 \, |Z_t(w)|^2}.
\end{eqnarray*}
\[  \p_t \left[|Z_t(z) - Z_t(w)|\,
      |Z_t(z) - Z_t(\overline w)| \right] = \hspace{1.5in}\]
\[
 \hspace{1in}
\left[|Z_t(z) - Z_t(w)|\,
      |Z_t(z) - Z_t(\overline w)| \right] \, \frac{2a\, X_t(z)\,
   X_t(w) }{|Z_t(z)|^2 \, |Z_t(w)^2|}. \]
Combining this with \eqref{jan15.1.alt} and the
stochastic product rule yields the following.
We will only use this lemma with $m=2$.

\begin{lemma} \label{jan17.lemma}
 Suppose $r \in \R$, $z_1,\ldots,z_m \in \Half$,
and $N_t(z_j)$ denotes the martingale in \eqref{martingale}.
Let
 \[  N_t =
N_t(z_1,\ldots,z_m) =\hspace{1.3in}\]
\begin{equation}  \label{jan17.2}
\left[\prod_{j=1}^m N_t(z_j) \right]
   \; \left[\prod_{j \neq k} |Z_t(z_j) - Z_t(
  z_k )| |Z_t(z_j)
    - Z_t(\overline z_k)|\right]^{-\frac {r^2\kappa} 4}.
\end{equation}
Then $N_t $ is a martingale satisfying
\[   dN_t = r\, N_t \, \left[\sum_{j=1}^m
   \frac{X_j(z_j)}{|Z_j(z_j)|^2} \right] \,
   dB_t.\]
\end{lemma}

\subsection{First moment}  \label{firstsec}

The proof of Theorem \ref{firsttheorem} relies
on the following estimate that can be found in
\cite[Theorem 9.1]{Lmulti}.  Since it will not
require much extra work here, we will also give
a proof in this paper.  Unlike
the second moment estimates, there is no need to restrict
this to $D \in \domain$.

\begin{lemma}   Suppose $\kappa < 8$ and
$\hat u > 2- \frac \kappa 8$.
Then there exists
  $c$ such that for all $x,y$ and $s \geq y$,
\begin{equation}  \label{firstmom}
  \E\left[|h'_{s^2}(z)|^d\right]
  \leq c \, s^{d-2} \, |z|^{2-d}  \, [\sin \theta_z]^{2-d-\hat u}  .
\end{equation}
\end{lemma}

\begin{proof}  See Lemma \ref{aug10.lemma1}.
Note that scaling implies
that it suffices to prove the result for $y_z=1$.\end{proof}

\begin{proof}[Proof of Theorem \ref{firsttheorem} given
\eqref{firstmom}]
If $\kappa < 8$, we can find
 $\hat u$ satisfying
\begin{equation} \label{uhat}
2 - \frac \kappa 8  < \hat u <  \frac 8 \kappa.
\end{equation}
Then,
\begin{eqnarray*}
 \E[I_{s^2}] &    \leq & \int_\Half \E\left[
  |h_{s^2}'(z)|^d \right] \, d\mu(z)\\
 & \leq & c \,  s^{d-2} \int_\Half
 |z|^{2-d}  \, [\sin \theta_z]^{2-d-\hat u}  \,
  |z|^{d-2} \, [\sin \theta_z]^{\frac \kappa 8 +
 \frac 8 \kappa -2} \, e^{-\beta |z|^2} \, dA(z)\\
& \leq & c s^{d-2}.
\end{eqnarray*}
The last inequality uses $\hat u < 8/\kappa$.
\end{proof}

\subsection{Proof of Theorem \ref{mainestimate}}

The martingale in \eqref{jan17.2}
 yields a simple
two-point estimate for the derivatives.
  This bound is not always sharp,  but
it suffices for proving
Theorem \ref{mainestimate}.

\begin{proposition}  \label{prop22}
 For every $m <\infty$, there exists
$c = c_m$ such that if
$ D \in \domain_m$,  $z,w \in \Half$, $s,t >0$,
\begin{equation}  \label{estimate}
 F_{sD}(z,w;ts^2) \leq c \, s^{ \frac {3\kappa}4-2}
  \, y_z^{-\frac \kappa 8} \,|z| \,  y_w^{-\frac \kappa
8} \, |w| \, |z-w|^{-\frac \kappa 4} \,
  |z-\overline w|^{-\frac \kappa 4}.
\end{equation}
\end{proposition}

\begin{proof}  By scaling we may assume $s=1$.
All constants in this proof
depend on $m$ but not otherwise on $D$.  Let
$N_t$ be the martingale from Lemma \ref{jan17.lemma}
with $r=1$.
 Then
\[   \E[N_t] =
 N_0 = y_z^{-\frac  \kappa 8} \,|z|\,
   y_w^{-\frac  \kappa 8} \, |w| \,  |z-w|^{-\frac \kappa 4}
  \, |z - \overline w|^{-\frac \kappa 4}. \]
If $\I_t(z,w;D) = 1,$   we have
\[   Y_t(z)^{-\frac \kappa 8} \, |Z_t(z)| \,
   Y_t(w)^{-\frac \kappa 8} \, |Z_t(w)| \geq c_1 > 0, \]
\[    |h_t(z) - h_t(w)| \,
 |h_t(z) - h_t(\overline w)| \leq c_2 < \infty.\]
Therefore,
\begin{equation}  \label{lose}
 |h_t'(z)|^d \, |h_t'(w)|^d
  \, \I_t(z,w;D)
    \leq c_3 \, N_t,
\end{equation}
and
\[\E \left[ |h_t'(z)|^d \, |h_t'(w)|^d
  \, \I_t(z,w;D)
\right] \leq c \, \E[N_t]
  . \]
\end{proof}

\begin{remark}  The estimate \eqref{estimate}
is not sharp if $z$ and $w$ are
close in the hyperbolic metric.  For example, suppose
 that
$z=2w =   \epsilon i $.
Then using distortion estimates we see that
\[    |h_t(z) - h_t(w)| \asymp |h_t'(z)| \, |z-w| \asymp
   \epsilon \, |h_t'(z)| . \]
On the event $\I_t(z,w;D)= 1$,
\[   N_t \asymp \epsilon^{-\frac \kappa 4}
 \, |h_t'(z)|^{2d - \frac \kappa 4}. \]
Typically, $|h_t'(z)| \ll \epsilon^{-1}$, so the estimate
\eqref{lose} in the proof is not sharp.
\end{remark}

\begin{proposition}
If $\kappa < 4$, then for every positive integer $m$
there exists  $c = c_m$ such that if
  $D \in \domain_m$
and $s \geq 1$,
$t
\geq 1$, $r
> 0$
\begin{equation}  \label{jan17.6}
    \int_\Half \int_\Half
 F_{tD}( z, w,rt^2)
  \, d\mu( w) \, d\mu_{s^2} ( z)  \leq c \,
(s/t)^{2-\frac {3\kappa} 4}
 \, s^{\frac \kappa 4} .
\end{equation}
\end{proposition}

 \begin{proof}   As in the previous proof, constants
in this proof may depend on $m$ but not otherwise
on $D$.
  By \eqref{estimate}
the left-hand side of \eqref{jan17.6}
 is bounded above by a
constant times
\[    \, t^{ \frac {3\kappa}4-2}\int_\Half \int_\Half
  \, y_z^{-\frac \kappa 8} \,|z| \,  y_w^{-\frac \kappa
8} \, |w| \, |z-w|^{-\frac \kappa 4} \,
  |z-\overline w|^{-\frac \kappa 4}\, d\mu(w) \,
  d\mu_{s^2}(z) . \]  Hence it suffices to show that
there exists $c$ such that for all $s$,
\begin{equation}  \label{prop23}
\int_\Half \int_\Half
  \, y_z^{-\frac \kappa 8} \,|z| \,  y_w^{-\frac \kappa
8} \, |w| \, |z-w|^{-\frac \kappa 4} \,
  |z-\overline w|^{-\frac \kappa 4}\, d\mu(w) \,
  d\mu_{s^2}(z) < c \, s^{2-\frac \kappa 2}.
\end{equation}
Recall from \eqref{phi3} that
\begin{equation}  \label{jan17.4}
  d\mu_{s^2}(z) \leq c \, e^{-\beta (|z|/s)^2} \,
         \,
  \,  y_z^{\frac \kappa 8
 + \frac 8 \kappa -2} \, |z|^{1 - \frac 8 \kappa }\;dA(z).
\end{equation}
We write the integral in \eqref{prop23}
as
\[   \int_\Half \Phi(z) \,  y_z^{-\frac \kappa 8} \,|z|
\, d\mu_{s^2}(z) , \]
where
\[ \Phi(z)  = \int_\Half
 y_w^{-\frac \kappa
8} \, |w| \, |z-w|^{-\frac \kappa 4} \,
  |z-\overline w|^{-\frac \kappa 4}\, d\mu(w) .\]
We will show that
\begin{equation} \label{jan15.13}
         \Phi(z) \leq c
  \,  |z|^{ -\frac \kappa 2}
\end{equation}
Using this and \eqref{jan17.4},
the integral in \eqref{prop23} is bounded
above by a constant times
\begin{eqnarray*}
   \int_{\Half}  |z|^{2-  \frac \kappa 2
 - \frac 8 \kappa}
\,
 e^{-\beta (|z|/s)^2}
 \,
 y_z^{  \frac 8 \kappa - 2}
 dA(z) & \leq &  \int_\Half |z|^{ - \frac \kappa 2
 }
\,
 e^{-\beta (|z|/s)^2}
 \, dA(z)\\
& =&  s^{2 -\frac \kappa 2}
  \int_0^\infty r^{1 - \frac \kappa 2}\,
  e^{-\beta r^2} \, dr .
\end{eqnarray*}
Hence it suffices to prove \eqref{jan15.13}.

Using \eqref{phi2},
we see that  $\Phi(z)$ is bounded by a constant times
\[  \Phi^*(z):=  \int_\Half K(z,w) \, dA(w) , \]
where
\[ K(z,w) =  1\{{0 < y_w < 2a}\} \,
   y_w^{\frac 8 \kappa
-2} \, |w|^{2 - \frac 8 \kappa}
  \, |z-w|^{-\frac \kappa 4} \,
  |z-\overline w|^{-\frac \kappa 4}\, e^{-\beta
 x_w^2}
   .\]
We write $\Phi^*(z) = \Phi_1(z) +\Phi_2(z)  +
 \Phi_3(z) $ where
\[  \Phi_1(z) = \int_{
|z-w| \leq y_z/2}
  K(z,w) \, dA(w)  ,\]
\[  \Phi_2(z) = \int_{ y_z/2 < |z-w|
 < |z|/2}
K(z,w) \,dA(w)  .\]
\[ \Phi_3(z) = \int_{
|z-w| \geq |z|/2}
 K(z,w)  \,dA(w)  .\]

If $|z-w| \leq y_z/2$ and $y_w \leq 2a$, then
 \[ y_w \leq 2a,\;\;
|w| \asymp |z|, \;\;
  y_w \asymp y_z, \;\;
 |z-w|^{-\frac \kappa 4} \,
  |z-\overline w|^{-\frac \kappa 4} \asymp
    |z-w|^{-\frac \kappa 4} \, y_z^{-\frac
 \kappa 4} .\]
Also,
\[  x_w^2 \geq |w|^2 - (2a)^2 \geq \frac{|z|^2}{4}
   - (2a)^2.\]
Hence
\[\Phi_1(z) \leq c \, e^{-\beta |z|^2/4} \,
|z|^{2 - \frac 8 \kappa}
  \, y_z^{\frac 8 \kappa - \frac \kappa 4 -2}
  \, \int_{|w-z| < y_z/2} |z-w|^{-\frac \kappa 4}
  \, dA(z) . \]
Therefore,
\begin{eqnarray*}
\Phi_1(z)  &  \leq &
 c \, e^{-\beta |z|^2/2} \,
|z|^{2 - \frac 8 \kappa}
  \, y_z^{\frac 8 \kappa - \frac \kappa 2}\, 1\{y_z \leq 4a\}\\
&  \leq & c\,  |z|^{2 - \frac \kappa 2}\, e^{-\beta|z|^2/2}
  \leq c\, |z|^{- \frac \kappa 2} .
\end{eqnarray*}

Suppose $|z-w| \geq |z|/2$.  Then
$|z-w| \asymp $ $|z|$  for  $|w| \leq 2|z|$,  and
$|z-w| $ $ \asymp$ $
|w|$ for $|w| \geq 2|z|$.
Using $\kappa < 8$,
\[ \int_{y_w < 2a, |w-z| \geq |z|/2, |w|
  \leq 2|z|}
y_w^{\frac 8 \kappa
-2} \, |w|^{2 - \frac 8 \kappa}
  \, |z-w|^{-\frac \kappa 2} \,
   e^{-\beta
 x_w^2}
   \, dA(w)
 \leq   \]
\[  c\, |z|^{-\frac \kappa 2}   \int_{y_w < 2a,
|w| \leq 2|z|}
 [\sin \theta_w]^{\frac 8 \kappa - 2}
  \, e^{-\beta x_w^2}  \, dA(w)
\leq c\, |z|^{- \frac \kappa 2} \,[|z|^2\wedge 1]
.  \]
\[ \int_{y_w < 2a, |w| > 2|z|}
y_w^{\frac 8 \kappa
-2} \, |w|^{2 - \frac 8 \kappa}
  \, |z-w|^{-\frac \kappa 2} \,
   e^{-\beta
 x_w^2}
   \, dA(w)
 \leq \hspace{1in} \]
\[  c e^{-\beta |z|^2}  \int_{y_w \leq 2a,
|w| > 2|z|} |w|^{-\frac \kappa 2}
 [\sin \theta_w]^{\frac 8 \kappa - 2}
 \, e^{-\beta x_w^2/2}  \, dA(w)
\leq c\, |z|^{- \frac \kappa 2}  .\]
Therefore, $\Phi_3(z) \leq c \, |z|^{-\frac
 \kappa 2}.$

We now consider $\Phi_2(z)$ which is bounded above
by a constant times
\[  \int_{y_w \leq 2a, \; y_z/2 < |z-w| < |z|/2}
   y_w^{\frac 8 \kappa - 2} \, |w|^{2 - \frac 8 \kappa}
 \, |z-w|^{- \frac \kappa 2} \, e^{-\beta x_w^2} \,
  d A(w) . \]
Note that for $w$ in this range,
$x_w^2 \geq |w|^2 - (2a)^2 \geq  (|z|/2)^2
 - (2a)^2$  and $|w| \asymp |z|$
and hence we can bound this by
\[  c e^{-\beta|z|^2/4} \, |z|^{2
 - \frac 8 \kappa}
 \int_{y_w \leq 2a, \; y_z/2 < |z-w| < |z|/2}
   y_w^{\frac 8 \kappa - 2}
 \; |z-w|^{- \frac \kappa 2} \,
  d A(w) . \]
The change of variables $w \mapsto w - x_z$ changes
the integral to
\[ \int_{y_w \leq 2a, \; y_z/2 < |w-iy_z| < |z|/2}
   y_w^{\frac 8 \kappa - 2}
 \, |w-iy_z|^{- \frac \kappa 2} \,
  d A(w) . \]
We split this integral into the integral over $|w| \leq
2 y_z$ and $|w| > 2 y_z$.  The integral
over $|w| \leq 2 y_z$ is bounded by a constant
times
\[ y_z^{-\frac \kappa 2} \, \int_{|w| \leq
  2 y_z} \, y_w^{\frac 8 \kappa - 2}
  \leq c\, y_z^{\frac 8 \kappa - \frac \kappa 2}
  \leq c \,|z|^{\frac 8 \kappa - \frac \kappa 2}. \]
The integral over $|w| > 2y_z$ is bounded by a constant
times
\[  \int_{y_w \leq 2a, 2y_z \leq |w| \leq
|z|/2}
    y_w^{\frac 8 \kappa - 2} \, |w|^{- \frac \kappa 2}
\, dA(w) \leq c
 \,|z|^{\frac 8 \kappa - \frac \kappa 2}. \]
We therefore get
\[   \Phi_2(z) \leq c \, e^{-\beta|z|^2/4} \,
  |z|^{2- \frac 8 \kappa} \,
|z|^{\frac 8 \kappa - \frac \kappa 2}
\leq
 c \, |z|^{2 - \frac \kappa 2} \,e^{-\beta |z|^2/4}
 \leq c \, |z|^{-\frac \kappa 2}.\]

\end{proof}

\subsection{Second moment}

\begin{lemma}  \label{apr13.lemma1}
If $ \kappa < 8$ and $m < \infty$, there
exist $c = c_m$ such that if  and $D \in \domain_m$,
  $z,w \in \Half$
and $2at \geq y_z,y_w$,
\begin{equation}  \label{apr11.1}
F_{tD}(z,w,  t^2)
    \leq
c \, t^{-\zeta} \,
  |z|^{\frac \zeta 2} \,
  |w|^{ \frac \zeta 2 }\,
 [\sin \theta_z]^{\frac \zeta 2- \frac 14 - \frac 2\kappa}
\;
[\sin \theta_w]^{\frac \zeta 2- \frac 14 - \frac 2 \kappa} ,
\end{equation}
where
 \begin{equation}  \label{zeta.2}
  \zeta  =
  \frac 4 \kappa - \frac{3\kappa}{16} - 1.
 \end{equation}
\end{lemma}

\begin{proof}  By the Cauchy-Schwarz inequality, it suffices
to prove the result for $z=w$ and by scaling we may assume
$y_z=1$.
Therefore, it suffices to prove
\[  F_D(x+i,x+i,t^2) \leq c \, t^{-\zeta}
  \, (x^2 + 1)^{\frac 14 + \frac
 2 \kappa} , \;\;\;\;\; t \geq 1/2a. \]
We let $z = x+i$ and write $Z_t = X_t + i Y_t= h_t(z) - U_t.$
Consider the martingale $N_t = N_t(z)$ as in \eqref{martingale}
with
\begin{equation}  \label{arbitrary}
  r = \frac 4 \kappa + \frac 12, \;\;\;\; \lambda = \frac
  2 \kappa  + \frac {3\kappa}
  {32} + 1 , \;\;\;\; r - \frac{\kappa r^2}{8} =
  \lambda - \frac{\kappa \, r}{4} =
\frac 2 \kappa - \frac \kappa {32}.
\end{equation}
Since $\E[N_{t^2}] = M_0$ and $Y_{t^2},
|Z_{t^2}|  \asymp
t$ when $\I_{t^2}(z;tD) = 1$,
\[  \E\left[|h_{t^2}'(z)|^{ \frac
  2 \kappa  + \frac {3\kappa}
  {32} + 1 }
   \;\I_{t^2}(z;tD)  \right]
  \leq  c\,t^{\frac \kappa {32} - \frac
  2 \kappa }
  \, (x^2 + 1)^{\frac 2 \kappa + \frac 14}. \]
>From \eqref{devbound}, we know that
\[   |h_{t^2}'(z)| \, \leq \sqrt{2at^2 + 1 }
  \leq c \, t.\]
Note that
\[  2d -  \left( \frac
  2 \kappa  + \frac {3\kappa}
  {32} + 1 \right) =  1 - \frac 2 \kappa  + \frac{5
 \kappa }{32} . \]
Hence
\begin{eqnarray*}\E\left[|h_{t^2}'(z)|^{2d}
   \,  \I_{t^2}(z;tD)\right]
& \leq & c \, t^{ 1 - \frac 2 \kappa  + \frac{5\kappa }{32} } \,
\E\left[|h_{t^2}'(z)|^{\frac
  2 \kappa  + \frac {3\kappa}
  {32} + 1 }
   \,  \I_{t^2}(z;tD) \right]\\
 & \leq & c \, t^{-\zeta} \, (x^2 + 1)^{\frac
 2 \kappa + \frac 14 }.
\end{eqnarray*}
\end{proof}

\begin{remark}  We have not given the motivation for
the choice \eqref{arbitrary}.  See \cite{Lmulti} for a
discussion of this.
\end{remark}

\begin{remark}
The   estimate \eqref{apr11.1}
for $z \neq w$ makes use of the
Cauchy-Schwarz inequality
\[ \left(\E\left[|h_t'(z)|^d \, |h_t'(w)|^d \,\I_{t}(z,w;D)
\right]\right)^2
   \leq \hspace{1in} \]
\[ \hspace{1in} \E\left[|h_t'(z)|^{2d}\, \I_t(z;D)\right]
  \, \E\left[|h_t'(w)|^{2d}\, \I_t(w;D)\right].\]
If $z$ and $w$ are close (for example, if $w$ is in the disk of radius
$\Im(z)/2$ about $z$), then the distortion theorem tells
us that $|h_t'(w)| \asymp
|h_t'(z)|$ and then the two sides of the
inequality agree up to
a multiplicative constant. However, if $z,w$ are far apart
(in the hyperbolic metric), the right-hand side can be much
larger than the left-hand side.
  Improving this estimate for $z,w$ far apart is
the key for proving good second moment bounds.
\end{remark}

The next lemma proves the $s=0$
case of \eqref{aug10.11.2}.  A similar argument
proves \eqref{aug10.11.2} for all $0 \leq s \leq
3(1+a)$, so in the next section we can restrict
our consideration to $s \geq 3(1+a)$.

\begin{lemma} \label{66} If $ \kappa < \kappa_0,$ there is a $c< \infty$
such that for all $t \geq 1$,
\[  \int_\Half \int_\Half F(z,w,t^2) \, d\mu(z) \, d\mu(w)
 \leq c \, t^{-\zeta}.\]
\end{lemma}

\begin{proof}

Since $t \geq 1$,
\eqref{apr11.1} gives
\[      F(z,w,t^2) \leq
c \, t^{-\zeta} \,
  |z|^{\frac \zeta 2} \,
  |w|^{ \frac \zeta 2 }\,
 [\sin \theta_z]^{- \frac{3 \kappa}{32} - \frac
 34}
\,
[\sin \theta_w]^{- \frac{3 \kappa}{32} - \frac
 34} .\]
Hence by \eqref{phi3} it suffices to show that
\[ \int_\Half  |z|^{\frac \zeta 2} \,
[\sin \theta_z]^{- \frac{3 \kappa}{32} - \frac
 34}
\; [\sin \theta_z]^{\frac \kappa 8 + \frac
 8 \kappa - 2} \, |z|^{\frac  \kappa 8-1}
  \, e^{-\beta|z|^2} \, dA(z) < \infty.\]
This will be true provided that
\[  - \frac{3 \kappa}{32} - \frac
 34 + \frac \kappa 8 + \frac
 8 \kappa - 2  > -1, \]
which holds for $\kappa < \kappa_0$
(see \eqref{kappa}). \end{proof}

\subsection{The correlation}

In this section, we state the  hardest estimate and then
show how it can  be used to prove the main result.
It will be useful to introduce some notation.
For $s \geq 3(1+a)$,
 let
\[  v(w,s) = v_m(w,s) = s^{2-d - \frac \zeta 2}  \,
\sup \, \left[t^\zeta \, y_z^{-\frac \zeta2 } \,
  \, [\sin \theta_z]^{ \frac 14 + \frac 2 \kappa} \, F_{tD}(w,z,t^2)\right]
 , \]
where
the supremum is over all $D \in \domain_m$, $t \geq 2s$ and all
$z\in \Half$ with
$|z|  \geq 3(1+a)s$.
In other words, if
$t \geq |z| \geq 3(1+a)$,
\begin{eqnarray}
  F_{tD}(w,z,t^2) & \leq   &  c
 \, t^{-\zeta} \,|z|^{d-2+
\frac \zeta 2} \, y_z^{\frac \zeta 2}
 \, [\sin \theta_z]^{-\frac 1 4 - \frac
   2 \kappa}\, v(w,|z|)\nonumber \\
  & = & c \, t^{-\zeta}
  \, [\sin \theta_z]^{\frac \zeta 2 -
 \frac 14 - \frac 2 \kappa} \, |z|
^{d-2+ \zeta} \, v(w,|z|).
  \label{435}
 \end{eqnarray}
The main estimate is the following.  The hardest part,
\eqref{aug11.1}, will be proved in the next subsection.

\begin{proposition}  \label{prop.aug1}
If $\kappa < \kappa_0$, there exists $u< \frac 8
 \kappa $  such that for each $m$ there
exists  $c< \infty$
 such that for all $s \geq 3(a+1)$ and $w
\in \Half$ with $y_w\leq 2a$,
\begin{equation}  \label{aug11.1}
   v(w,s)
\leq c \, [\sin \theta_w]^{2-d-u} \, |w|^{2-d}
 = c [\sin \theta_w]^{1 - \frac \kappa 8-u} \, |w|^{2-d}
 .
\end{equation}
In particular,
\[   \int v(w,s) \, d\mu(w) < c , \]
and hence if $D \in \domain_m$ and
$z \in \Half$,
\begin{equation}   \label{aug9.5}
 \int  F_{tD}(w,z,t^2) \, d\mu(w) \leq c  \, t^{-\zeta}\,
  [\sin \theta_z]^{\frac \zeta 2 -
 \frac 14 - \frac 2 \kappa} \, |z|
^{d-2+ \zeta} .
\end{equation}
\end{proposition}

\begin{proof}  We delay the proof
of \eqref{aug11.1} to Section \ref{hardsec}, but we will
show here how it implies the other
two statements.
Using \eqref{phi3}, we have
\[
   \int v(w,s) \, d\mu(w)  \leq  \hspace{2in}\]
\[
  c \, \int [\sin \theta_w]^{1- \frac \kappa 8 -u} \, |w|^{2-d}
 \, |w|^{d-2} \, [\sin \theta_w]^{\frac \kappa 8 +
 \frac 8 \kappa -2} \, e^{-\beta |w|^2} \, dA(w) < \infty.
\]  The last inequality uses $u < 8/\kappa$.
The estimate \eqref{aug9.5}  for $|z| \geq 3(1+a)$
follows
immediately from \eqref{435}; for other $z$ it is proved
as in Lemma \ref{66}.
\end{proof}

\begin{corollary} If $\kappa < \kappa_0$, then for every
$m$ there is a $c$ such that if $D \in \domain_m$,
$s^2 \geq 1, t^2 \geq 1$.
\[ \sum_{j=0}^{s^2-1}
\int_\Half \int_\Half F_{tD}(z,w,j + t^2,t^2) \, d\mu(w)
  \, d\mu(z) \leq
   c \, (t/s)^{-\zeta} s^{2(d-1)}. \]
\end{corollary}

\begin{proof}[Proof assuming Proposition \ref{prop.aug1}]
>From Lemma \ref{aug11.lemma4} and \eqref{phi3}, we know that
\[ \sum_{j=0}^{s^2-1}
\int_\Half \int_\Half F(z,w,j + t^2,t^2) \, d\mu(w) \,
 d\mu(z)
\leq \hspace{1in}\]
\[c \int_\Half
 \left[\int_\Half  F(z,w,t^2)  \, d\mu(w) \right]
    \,[\sin \theta_z]^{\frac \kappa 8
  + \frac 8 \kappa - 2} \, |z|^{d-2} \,
 e^{-\beta |z|^2/s^2}
\,dA(z) .\]
Using the previous lemma and estimating as in Lemma
\ref{66}, we see that for $\kappa < \kappa_0$ this
is bounded by a constant times
\[  t^{-\zeta}
 \int_{\Half}|z|^{\zeta + 2(d-2)} \, e^{-\beta
  |z|^2/s^2} \, dA(z) = c \,t^{-\zeta} \,  s^{\zeta + 2d-2}.\]
 \end{proof}

\subsection{Proof of \eqref{aug11.1}}  \label{hardsec}

It was first observed
 in \cite{RS}  that
when studying moments of $|h_t'(z)|$ for a fixed $z$ it is
useful to consider a  paraametrization such
that $Y_t(z)$ grows deterministically. The next
lemma uses this reparametrization to get a result about fixed
time.  The idea is to have a stopping time in the new
parametrization that corresponds to a bounded stopping
time in the original parametrization.  A version of
this stopping time appears in \cite{Lmulti} in the
proof of the first moment estimate.
If $\kappa < \kappa_0$, there exists
 $u$ satisfying
\begin{equation}  \label{uvalue}
              \frac {7}4 - \frac{\kappa}
 {32} < u < \frac 8 \kappa.
\end{equation}
For convenience, we fix one such value of $u$.
Let $N_t(w)$ be the martingale from \eqref{martingale}
which we can write as
\[  N_t(w) = |h_t'(w)|^d \, Y_t(w)^{2-d} \,
         [R_t(w)^2 + 1]^{\frac 12},\;\;\;\;\; R_t(w) = X_t(w)/
  Y_t(w), \]
and recall $\hat u$ from \eqref{uhat}.

\begin{lemma} \label{aug10.lemma1}
 if $a > a_0$,
there exists $c$ such that the following is true.  For
each $t$ and each $w=x+yi$ with $y \leq t$,
there exists a stopping time
$\tau$ such that
\begin{itemize}
\item  \[ \tau \leq t^2,\]
\item
\begin{equation}  \label{aug10.1}
|U_s| \leq (a+2) \, t, \;\;\;\; 0 \leq s \leq \tau,
\end{equation}
\item
\begin{eqnarray*}
 \E\left[|h_{\tau}'(w)|^d \,
   Y_\tau^{\frac \zeta 2} \, (R_\tau^2 + 1)
 ^{\frac 18 + \frac 1 \kappa } \right] & =&
 \E\left[N_\tau \, Y_\tau^{\frac \zeta 2 + d-2}
  \, (R_\tau^2 + 1)^{\frac a 2 - \frac 38} \right] \\
 & \leq &  c\, (t+1)^{d-2+ \frac \zeta 2}
 \, |w|^{2-d} \,[\sin \theta_w]^{2-d -u}  .
\end{eqnarray*}
\end{itemize}
Here $N_s = N_s(w), Y_s = Y_s(w), R_s = X_s(w)/Y_s(w)$.

Moreover,
   if $a > 1/4$, there exists $c$ such that
\begin{equation} \label{aug3.4}
\E\left[|h_{t^2}'(w)|^{d} \right] \leq c\,
[(x/y)^2+1]^{\frac {\hat u}{2}} \,
   \left(\frac{t}{y} \vee 1 \right)^{d-2}.
\end{equation}
\end{lemma}

\begin{proof} By scaling, it suffices to prove
the lemma for $y = 1$, i.e., $w = x + i.$
Without loss of generality, we assume $x \geq 0.$
If $t \leq 1$, we can choose
the trivial stopping time $\tau \equiv 0$ and
\eqref{aug3.4} is easily derived from the Loewner equation.
Hence we may assume $t \geq 1$.
We write $t= e^{al}, x = e^{am}$.
For notational ease we will assume that $l,m$ are
  integers, but it is easy to adjust
the proof for other $l,m$. We will define the stopping
time for all $a > 1/4$; it will be used for proving
\eqref{aug3.4}.

We  consider
a  parametrization in which the
logarithm of the  imaginary
part grows linearly.
Let
\[  \sigma(s) = \inf\{u: Y_u = e^{as} \},\;\;\;\;
\hat X_s = X_{\sigma(s)},\;\;\;
  K_s = R_{\sigma(s)} = e^{-as} \, \hat X_{s} , \]
and note that $\hat Y_s = Y_{\sigma(s)} = e^{as}.$
Using the Loewner equation, we can see that
\[  \p_s\sigma(s) = \hat X_s^2 + \hat Y_s^2
   = e^{2as} \, (K_s^2 + 1). \]
Let $\hat N_s = N_{\sigma(s)}(x+i)$,
\[  \hat N_s  = \hat N_s(x+i) =
    |h_{\sigma(s)}'(x+i)|^d \, e^{(2-d)as} \,
         (K_s^2 + 1)^{\frac 12}. \]
Since $\hat N_s$ is a time change of a martingale,
it is easy to see that it is a martingale.
Note that $K_0 = x=e^{am}$.

We first define our stopping time in terms of the
new parametrization.   Let $\rho$ be  the smallest $r$
such that
\[     \hat X_r^2 + \hat Y_r^2 \geq
   \frac{e^{2al}}{(l-r+1)^4} , \;\;\;
\mbox{i.e., } \;\;\;\;
\sqrt{K_r^2 + 1} \geq \frac{e^{a(l-r)}}{(l-r+1)^2} . \]
One can readily check that the following properties hold.
   \[\rho \leq l, \]
\[  \rho = 0 \mbox { if } m \geq l , \]

   \[ \hat X_s^2 \leq \hat X_{s}^2  + \hat  Y_{s}^2
     \leq  \frac{e^{2al}}{(l-s + 1)^4} , \;\;\;\;
   0 \leq s \leq \rho,\]
\[  \sigma(\rho) = \int_0^\rho e^{2as}
  \, [K_s^2 + 1] \, ds  \leq \int_0^l \frac{e^{2al}}
   {(l-s+1)^4} \, ds  \leq e^{2al}, \]
 \[\int_0^\rho |\hat X_r| \, dr \leq \int_0^
   l \frac{e^{al}}
   {(l-s+1)^2} \, ds  \leq e^{al}   . \]

We define
\[ \tau = \sigma(\rho)  \leq e^{2al} ,\] i.e., $\tau$
is essentially the same
stopping time as $\rho$
except using the original parametrization.
Note that
\[  \int_0^\tau \frac{|X_t|}{X_t^2 + Y_t^2} \, dt
   = \int_0^\rho |\hat X_t| \, dt \leq e^{al}. \]

Recall that
\[  dX_s =  \frac{a \, X_s}{X_s^2 + Y_s^2} \, ds - dU_s, \]
which implies
\[    -U_s = (X_s-X_0) - \infty_0^s\frac{a \, X_r}{X_r^2 + Y_r^2} \, dr .\]
If $X_0 \geq e^{al}$, then $\tau =0$ and \eqref{aug10.1}
holds immediately.  Otherwise,
\[ |U_t| \leq |X_t| + |X_0| + a \int_0^\rho \frac{|X_s|}
    {X_s^2 + Y_s^2} \, ds \leq (2+a) \, e^{al}. \]
This gives \eqref{aug10.1}.

 Let $A_j$ be the event
\[ A_j = \{t-j < \rho \leq t-j+1\} =
    \{e^{a(t-j)} < Y_{\tau}  \leq e^{a(t-j+1)} \}. \]
On the event ${A_j}$, we have
\[      Y_{\tau} \asymp e^{at} \, e^{-aj}, \;\;\;\;\;
       R_\tau^2 +1 \asymp e^{2aj} \, j^{-4} . \]
The Girsanov theorem
implies that
\[  \E[N_\tau 1_{A_j}] = N_0 \, \Prob^*(A_j)
  = (x^2 + 1)^{\frac 12} \,  \Prob^*(A_j)
 \asymp e^{am} \, \Prob^*(A_j)  , \]
where we use $\Prob^*$ to denote the probabilities given by
weighting by the martingale $N_t$.  We claim that there
exist $c,\beta$ such that
\begin{equation}  \label{girsanov}
 \Prob^*(A_j) \leq c \, j^\beta \, e^{(4a-1)(m-j)a}.
\end{equation}
To see this, one considers the process in the new
parametrization and notes that after weighting by the martingale
$\hat N$, $K_t$ satisfies
\begin{equation}  \label{aug10.4}
     dK_s = \left(1 - 2a \right) \, K_s \, ds
   + \sqrt{K_s^2 + 1} \, dW_s,
\end{equation}
where $W_s$ is a standard Brownian motion with $K_0 = x$.
Equivalently,
$K_s = \sinh J_s$ where $J_s$ satisfies
\begin{equation}  \label{aug10.5}
   dJ_s = -q\,
  \tanh J_s \, ds + dW_s,
\end{equation}
with $J_0 =  \sinh^{-1} x, q = \frac 12 - 2a$.
Standard techniques (see \cite[Section 7]{Lmulti})
show that $J_t$
is positive recurrent with invariant density
proportional to $[\cosh x]^{-2q}.$  If $ 0 < x < y$, then
the probability starting at $x$ of reaching $y$ before $0$
is bounded by $c \, [\cosh x/\cosh y]^{2q}.$   Using these
ideas, we get that for every $k$,
\[   \Prob^x\{ y \leq J_t \leq y+1 \mbox{ for some }
   k \leq t \leq k+1\} \leq c \, \left( \frac{\cosh x}{
  \cosh y} \right)^{2q}. \]

On the event $A_j$, we know that $Y_{t^2} \geq Y_\tau
\asymp e^{al} \, e^{-aj}.$
  The martingale property and \eqref{girsanov} imply that
\[ \E\left[N_{t^2} \, 1_{A_j}\right] =
  \E\left[N_\tau \, 1_{A_j} \right] = (x^2 + 1)^{\frac 12} \,  \Prob^*(A_j)
 \leq c \, e^{am} \, \left[1 \wedge
 j^{\beta}\, e^{(4a-1)(m-j)a}\right].\]
Therefore,
\begin{eqnarray*}
e^{-am} \, e^{al(2-d)}
\, \E\left[ |h_{t^2}'(z)|^{d} \, 1_{A_{j}}\right]
   & \leq & c\, e^{aj(2-d)} \, e^{-am} \,
     \E\left [N_{t^2} \, 1_{A_j} \right]\\
  & \leq & c  \, e^{aj(2-d)} \,\left[1 \wedge
 j^{\beta} \,  e^{(4a-1)(m-j)a}\right].
\end{eqnarray*}
\begin{eqnarray*}  \lefteqn{ e^{-am} \, e^{at(2-d)}
\, \E\left[ |h_{t^2}'(z)|^{d} \right]}\hspace{.5in}\\
  & \leq & c\sum_{j=1}^\infty e^{aj(2-d)}\,
 \left[1 \wedge
 j^{\beta} \,  e^{(4a-1)(m-j)a}\right] \\
 & \leq & c \left[
  \sum_{j=1}^m e^{aj(2-d)} +  e^{(4a-1)m}
\sum_{j=m+1}
^\infty  j^\beta \, e^{aj[(2-d) + 1-4a]}\right]\\
 & \leq & c \left[
   e^{am(2-d)} +  e^{(4a-1)m}
\sum_{j=m+1}
^\infty  j^\beta \, e^{aj[(2-d) + 1-4a]}\right]\\
 & \leq & c \, m^\beta \, e^{am(2-d)}. \end{eqnarray*}
The last inequality requires
\[ 2-d+ 1-4a  < 0, \]
which is readily checked for $a > 1/4$.
Therefore,
\begin{eqnarray*}
  \E\left[ |h_{t^2}'(z)|^{d} \right]
 &  \leq&  c \, t^{d-2} \, m^\beta \, e^{am(3-d)}
  \\& \leq & c \, t^{d-2} \, [\log(x^2 + 2)]^\beta \,
     (x^2 + 1)^{1 - \frac \kappa{16}}\\
 &\leq&  c \, t^{d-2} \, (x^2 + 1)^{\frac {\hat u}2}.
\end{eqnarray*}
This establishes \eqref{aug3.4}.

Note that
\[
\E\left[ N_\tau \, \left(R_\tau^2 + 1\right)^{\frac a 2
- \frac 38
  } \, Y_\tau^{\frac \zeta 2 + d-2}\, 1_{A_j} \right]
\asymp \hspace{1.5in} \]
\[ \hspace{1.5in}  j^{-2a + \frac 32} \, e^{aj(a - \frac 34)}
 \, e^{(l-j)a(\frac \zeta 2 + d - 2)} \,
  \E\left[N_\tau \, 1_{A_j}\right]. \]
Therefore,
\begin{eqnarray*}
  \lefteqn{ e^{-am} \, e^{-al(\frac \zeta 2 + d - 2)}
\, \E\left[ N_\tau \, \left(R_\tau^2 + 1\right)^{\frac a 2
- \frac 38
  } \, Y_\tau^{\frac \zeta 2 + d-2} \right]} \hspace{.4in} \\
& \leq &
   c \sum_{j=1}^{\infty} j^{-2a + \frac 32} \,
\, e^{aj(a - \frac 34)} \,  e^{-ja ( \frac{\zeta }{2}+ 2-d)}
 \, \left[1 \wedge
 j^{\beta}\, e^{(4a-1)(m-j)}\right] \\
& \leq & c \, m^{\beta} \, e^{ma(-\frac \zeta 2 + d-2+a -
\frac 3{4})} .
\end{eqnarray*}
The last inequality requires
 \[ -\frac{\zeta} 2+ 2-d  +a - \frac 34  + 1 - 4a < 0 . \]
Recalling that
\[   \frac{\zeta}{2} = a - \frac{3}{16a} - \frac 12, \]
this becomes
\[        -4a - \frac 1{16a} + \frac 7 4 < 0. \]
This is true if $\kappa < \kappa_0$, see \eqref{kappa}.

\end{proof}

\begin{proposition}  If $ a > a_0$,
for every $\epsilon > 0$
 there is a $c$ such
that the following is true.  Assume
\[  z =  x + i y, \;\;
  w =  \hat x + i \hat y \;\; \in \Half \]
with $\hat y \leq 2a + 1$ and
$|z| \geq 3(a+1)$.  Then for $t\geq 2 , s\geq 1$,
\begin{equation}  \label{jul27.1}
  F_{stD}(sz,w, (st)^2) \leq c \, t^{-\zeta}
  \, |z|^{\frac \zeta 2} \, [\sin \theta_z]
  ^{\frac \zeta 2 - \frac 14 - \frac 2 \kappa}
 \,
[\sin \theta_w]^{2-d- u}
  \,  (|w|/s)^{2-d}.
\end{equation}
\end{proposition}

\noindent {\bf Remark.}
If $|z|  \geq 3(a+1)s$ and $t \geq 2s$,  we can write
\eqref{jul27.1} as
\[  F_{stD}(z,w,t^2) \leq c \, (t/s)^{-\zeta}\, s^{d-2}\,
\, [\sin \theta_z]
  ^{\frac \zeta 2 - \frac 14 - \frac 2 \kappa}
 \,
[\sin \theta_w]^{2-d- u}
  \,  |w|^{2-d}.
\]
Therefore this proposition completes the proof
of \eqref{aug11.1}.

\begin{proof}

By scaling, $F_{stD}(sz,w,(st)^2) = F_{tD}(z,w/s,t^2)$;
hence, without loss of generality
 we may assume $s=1$.   We assume that
 $\tau$ is a stopping time as in the previous lemma
for $w$ and time $1$.  In particular, $0 \leq \tau \leq 1$.
We can find a domain $D'$ such that for all $1/2 \leq r \leq 1$,
$rD \subset D'$.

We write
$Z_s(z) = h_s(z) -U_s,  Z_s(w)
= h_s(w) - U_s$, etc. for the images under the flow.
By definition, $ F_{tD}(z,w,t^2) = \E[\Lambda]$
where    $\Lambda$
denotes the random variable
\[  \Lambda = \Lambda_D(z,w,t^2) :=
 \left|h_{t^2}'(z)\right|^d \, \left|h_{t^2}'(w)\right|^d
   \,\I_{t^2}(z,w;D), \]
and note that
\begin{equation}  \label{oct21.20}
 \E\left [ \Lambda \mid \G_\tau\right]
   =  |h_{\tau}'(z)|^{d} \, |h_{\tau}'(w)|^d
   \,    F_{tD}(Z_\tau(z),Z_\tau(w),t^2 - \tau).
\end{equation}

Since $|U_s| \leq 2+a$ for $s \leq \tau$, it
follows  from \eqref{reverse} that
\[   \p_s | h_s(z)|  \leq 1 , \;\;\;\;\;
   |h_s(z) - z| \leq s, \;\;\;\;
  s \leq \hat \tau, \]
where $\hat \tau$ denotes
  the minimum of $\tau $ and the
first time that $|h_s(z)| \leq 2 + 2a$.
Since $|z| \geq 3 + 3a$, this implies
the following estimates for $0 \leq s \leq \tau$:
\[  |h_s(z) - z| \leq 1, \;\;\;\;
   |h_s(z)| \geq 2 + 3a,
\;\;\; |U_s - h_s(z)| \geq a, \]
 \[  |Z_s(z) - z| \leq |h_s(z) - z| + |U_s| \leq 3 + a.
  \]
In particular, since $|z| \geq 3(1+a)$, there exists $c_1,c_2$
such that
\[  c_1\, (x^2 + 1) \leq X_s(z)^2 + 1
  \leq c_2 \, (x^2 + 1) , \;\;\;\; 0 \leq s \leq \tau, \]
\[    y_z \leq Y_s(z) \leq c_2 \, y_z, \;\;\;\; 0 \leq s \leq \tau.\]

We therefore get
\[   F_{tD}(Z_\tau(z),Z_\tau(w),t^2 - \tau) \leq
\sup
    F_{\tilde t D'}(\tilde z,Z_{\tau}(w), \tilde t^2) , \]
where the supremum is over all $t^2 -1 \leq \tilde t^2
 \leq t^2$ and all $\tilde z = \tilde x + i \tilde y$ with
\[           c_1
\, (x^2 + 1) \leq \tilde x^2 + 1 \leq c_2 \, (x^2+1) , \;\;\;\;
     y \leq
   \tilde y \leq c_2 \, y. \]
Using \eqref{apr11.1},  we get
\[  F_{tD}(Z_\tau(z),Z_\tau(w),t^2 - \tau)
  \leq \hspace{1.5in}\]
\[ \hspace{1.5in} c \, t^{-\zeta}
  \, |z|^{\frac \zeta 2} \,
 [\sin \theta_z]^{\frac \zeta 2-\frac 14 - \frac 2 \kappa
  } \, \left(R_\tau^2(w) + 1\right)^{\frac 1 8 + \frac
 1 \kappa
  }
 \,  {Y_\tau(w)} ^{ \frac \zeta 2}.\]
By differentiating \eqref{reverse}, we get
\[   |\p_s h_s'(z)| \leq a \, |h_s'(z)|,\;\; 0\leq s \leq \tau,\;\;\;\;
    |h_\tau'(z)| \leq e^a. \]
Therefore, plugging into \eqref{oct21.20}, we get
\[   \E\left[\Lambda \mid \G_\tau \right]  \leq
   c \, t^{-\zeta}
  \, |z|^{\frac \zeta 2} \,[\sin \theta_z]^{\frac \zeta 2-\frac 14 - \frac 2 \kappa
  } \,
|h_\tau'(w)|^{d} \,
\left(R_\tau^2(w) + 1\right)^{\frac 18 + \frac 1 \kappa
  } \,
 \, {Y_t(w)}  ^{ \frac \zeta 2}.\]
Taking expectations,
 and using the previous
lemma, we get
\[  F_{tD} (z,w,t^2)
 \leq
   c \, t^{-\zeta}
  \, |z|^{\frac \zeta 2}
 \,[\sin \theta_z]^{\frac \zeta 2-\frac 14 - \frac 2 \kappa
  } \,
\, |w|^{2-d} \, [\sin \theta_w]^{2-d-u}. \]

\end{proof}

\bibliographystyle{halpha}
\bibliography{scottfinal}

\end{document}